\documentclass[11pt,a4paper,reqno]{amsart}%
\usepackage{amssymb}
\usepackage{amsmath}
\usepackage{amsfonts}
\usepackage{color}
\usepackage[usenames,dvipsnames]{xcolor}
\usepackage{bbm}
\usepackage{graphicx}
\usepackage{rotating}
\usepackage{hyperref}
\usepackage{mathrsfs}
\usepackage[all]{xy}%
\providecommand{\U}[1]{\protect\rule{.1in}{.1in}}
\newtheorem{theorem}{Theorem}

\newtheorem{conjecture}[theorem]{Conjecture}
\newtheorem{corollary}[theorem]{Corollary}

\newtheorem{definition}[theorem]{Definition}

\newtheorem{lemma}[theorem]{Lemma}

\newtheorem{proposition}[theorem]{Proposition}
\newtheorem{remark}[theorem]{Remark}

\thanks{}
\email{lvitagliano@unisa.it}
\begin{document}
\title{On The Strong Homotopy Lie-Rinehart Algebra of a Foliation}
\author{\textsc{{Luca Vitagliano}}}
\address{DipMat, Universit\`a degli Studi di Salerno, {\& Istituto Nazionale di Fisica
Nucleare, GC Salerno,} Via Ponte don Melillo, 84084 Fisciano (SA), Italy.}

\begin{abstract}
It is well known that a foliation $\mathscr{F}$ of a smooth manifold $M$ gives
rise to a rich cohomological theory, its \emph{characteristic (i.e., leafwise)
cohomology}. Characteristic cohomologies of $\mathscr{F}$ may be interpreted,
to some extent, as functions on the space $\boldsymbol{P}$ of integral
manifolds (of any dimension) of the characteristic distribution $C$ of
$\mathscr{F}$. Similarly, characteristic cohomologies with local coefficients
in the normal bundle $TM/C$ of $\mathscr{F}$ may be interpreted as vector
fields on $\boldsymbol{P}$. In particular, they possess a (graded) Lie bracket
and act on characteristic cohomology $\overline{H}$. In this paper, I discuss
how both the Lie bracket and the action on $\overline{H}$ come from a strong
homotopy structure at the level of cochains. Finally, I show that such a
strong homotopy structure is canonical up to {canonical} isomorphisms.

\end{abstract}
\maketitle

\section{Introduction}

The space of leaves of a foliation is not necessarily a smooth manifold.
However, there exists a formal, cohomological way of defining a differential
calculus on it. Namely, a foliation is a special instance of a \emph{diffiety}%
. A diffiety (or a $D$\emph{-scheme}, in the algebraic geometry language) is a
geometric object formalizing the concept of \emph{partial differential
equation} (PDE). Basically, it is a (possibly infinite dimensional) manifold
$M$ with an involutive distribution $C$ (in the case when $M$ is finite
dimensional, $(M,C)$ is the same as a foliation of $M$). It emerges in the
geometric theory of PDEs as the infinite prolongation of a given system of
differential equations \cite{b...99}. Solutions (initial data, etc.) of a
system of PDEs with $n$ independent variables, correspond bijectively to
$n$-dimensional ($(n-1)$-dimensional, etc.) integral submanifolds of the
corresponding diffiety. Vinogradov developed a theory, {which is known as}
\emph{secondary calculus} \cite{v84,v98,v01}, formalizing in cohomological
terms the idea of a differential calculus on the space of solutions of a given
system of PDEs, or, which is roughly the same, the space of integral manifolds
of a given diffiety $(M,C)$. In other words, \emph{secondary calculus provides
substitutes for vector fields, differential forms, differential operators,
etc., on a (generically) very singular space where these objects cannot be
defined in the usual (smooth) way}.

{Namely, let $(\overline{\Lambda},\overline{d})$ be the differential graded
(DG) commutative agebra of differential forms on $M$ longitudinal along $C$
(i.e., the exterior algebra of the module of sections of the quotient bundle
$T^{\ast}M/C^{\bot}$, see Section \ref{HAF} for details) and $\overline
{\mathfrak{X}}$ the module of vector fields transversal to $C$ (i.e., the
module of sections of the quotient bundle $TM/C$, see again Section \ref{HAF}
for details). The $\overline{\Lambda}$-module, $\overline{\Lambda}%
\otimes\overline{\mathfrak{X}}$ possesses a differential $\overline{d}$ which
makes it a DG module over $(\overline{\Lambda},\overline{d})$. Now,
\emph{secondary} functions are just cohomologies of $(\overline{\Lambda
},\overline{d})$, and secondary vector fields are cohomologies of
$(\overline{\Lambda}\otimes\overline{\mathfrak{X}},\overline{d})$. Similarly,
\emph{secondary} differential forms, etc., are characteristic (i.e.,
longitudinal along $C$) cohomologies of $(M,C)$ (with local coefficients in
transversal differential forms, etc.).} All constructions of standard calculus
on manifolds ({Lie bracket of vector fields,} action of vector fields on
functions, exterior differential, insertion of vector fields in differential
forms, Lie derivative of differential forms along vector fields, etc.) have a
secondary analogue, i.e., a formal, cohomological analogue within secondary
calculus. {In the trivial case when $\dim C=0$, secondary calculus reduces to
standard calculus on the manifold $M$} (see the first part of \cite{v09} for a
compact review of \emph{secondary Cartan calculus}).

Since secondary constructions are algebraic structures in cohomology, it is
natural to wonder whether they come from algebraic structures
\textquotedblleft\emph{up to homotopy}\textquotedblright at the level of cochains.

The present paper is the first in a series aiming at exploring the following

\begin{conjecture}
\label{Conj} All secondary constructions come from suitable homotopy
structures at the level of (characteristic) cochains.
\end{conjecture}

A few instances motivating Conjecture \ref{Conj} are scattered through the
literature. Namely, Barnich, Fulp, Lada, and Stasheff \cite{b...98} proved
that a (secondary) Poisson bracket on the space of histories of a field theory
(which is nothing but the space of solutions of the trivial PDE $0=0$, whose
underlying diffiety is a \textquotedblleft free\textquotedblright\ one, i.e.,
an infinite jet space) comes from a (non canonical) $L_{\infty}$-structure at
the level of horizontal forms. Similarly, Oh and Park \cite{op05} showed that
the Poisson bracket on characteristic cohomologies of the degeneracy
distribution of a presymplectic form comes from an $L_{\infty}$-structure on
longitudinal forms. Finally, C.{} Rogers \cite{r10} showed that $L_{\infty}$
algebras naturally appear in multisymplectic geometry. More precisely, he
proved that Hamiltonian forms in multisymplectic geometry build up an
$L_{\infty}$-algebra (see also \cite{z10} for a generalization of the results
of Rogers to field theories with non-holonomic constraints). In fact, Rogers'
$L_{\infty}$-algebra induces the standard Lie algebra of conservation laws in
the characteristic cohomology of the \emph{covariant phase space} of a
multisymplectic field theory (see \cite{v09}). In its turn, such a Lie algebra
can be understood as a secondary analogue of the Lie algebra of first
integrals in Hamiltonian mechanics.

In this paper, I show that the Lie-Rinehart algebra of secondary vector fields
comes from a strong homotopy (SH) Lie-Rinehart algebra structure on the
corresponding cochains, i.e., transversal vector field valued longitudinal
forms. To keep things simpler, I assume $M$ to be finite dimensional. In fact,
all the proofs are basically algebraic and immediately generalize to the
infinite dimensional case.

I have to mention here that three papers already appeared containing results
closely related to results in this paper. Firstly, in \cite{h05} Huebschmann
shows that higher homotopies naturally emerge in the theory of characteristic
cohomologies of foliations. Specifically, he proposes a definition of
\textquotedblleft homotopy Lie-Rinehart algebra\textquotedblright, which he
calls \emph{quasi-Lie-Rinehart algebra}, and proves (among numerous other
things) that a (split) Lie subalgebroid in a Lie algebroid gives rise to a
quasi-Lie-Rinehart algebra. In fact, the homotopy Lie-Rinehart algebra
presented in this paper coincides with Huebschmann's quasi-Lie-Rinehart
algebra in the case of the Lie subalgebroid defined by a foliation. Indeed, a
quasi-Lie Rinehart algebra is a special type of SH Lie-Rinehart algebra, but
this is not explicitly stated by Huebschmann in his paper. In the subsequent
sections, I discuss the precise relation between quasi-Lie-Rinehart algebras
and SH Lie-Rinehart algebras, and clarify the novelty of the present paper
with respect to \cite{h05} (see Remark \ref{Hue2} of Section \ref{SHDGLR} and
last paragraph of Section \ref{HAF}). Secondly, very recently, Chen,
Sti\'{e}non, and Xu \cite{csx12} showed that the Lie bracket in the cohomology
of a Lie subalgebroid with values in the quotient module comes from a homotopy
Leibniz algebra at the level of cochains (see the end of Section \ref{HAF} for
a comparison between their results and results in this paper). Thirdly, when I
was preparing a revised version of my manuscript arXiv:1204.2467v1,
there appeared on the arXiv itself the paper \cite{j12} by Ji. Ji proves that
a (split) Lie subalgebroid in a Lie algebroid gives rise to an $L_{\infty}%
$-algebra. In fact, again in the case of the Lie subalgebroid defined by a
foliation, Ji's $L_{\infty}$-algebra can be obtained by the SH Lie-Rinehart
algebra of this paper forgetting about the anchors (see the first appendix for
the relation between Ji's construction and the construction in this paper).

The paper is organized as follows. It is divided in three parts. The first one
contains algebraic foundations and it consists of three sections. In Section
\ref{SHS}, I recall the definitions of (and fix the conventions about) SH
algebras (including their morphisms), SH modules and SH Lie-Rinehart algebras
(which, to my knowledge, have been defined for the first time {by Kjeseth} in
\cite{k01}). In Section \ref{SHDGLR}, I present in details the DG algebra
approach to SH Lie-Rinehart algebras which is dual to the coalgebra approach
of Kjeseth \cite{k01} (computational details, are postponed to Appendix
\ref{Appendix}). The algebra approach is, in my opinion, more suitable for the aims of this paper. Indeed, {the
existence of the SH Lie-Rinehart algebra of a foliation is an immediate
consequence of the existence of the exterior differential in the algebra of
differential forms on the underlying manifold (see Section \ref{SHLRF})}.

In Section \ref{MSHLR}, I use the DG algebra approach to discuss morphisms of
SH Lie-Rinehart algebras, over the same DG algebra. The second part of the
paper contains the geometric applications and it consists of five sections.
Section \ref{FVVF} reviews fundamentals of the Fr\"{o}licher-Nijenhuis
calculus on form-valued vector fields (more often named vector-valued
differential forms \cite{fn56}). The SH Lie-Rinehart algebra of a foliation
has a nice description in terms of Fr\"{o}licher-Nijenhuis calculus. In
Section \ref{HAF}, I briefly review the characteristic cohomology of a smooth
foliation, and state the theorem about the occurrence of a SH Lie-Rinehart
algebra in the theory of foliations. Section \ref{GSF} contains more
preliminaries on geometric structures over a foliated manifold. In Section
\ref{SHLRF}, I present the SH Lie-Rinehart algebra of a foliation and describe
it in terms of Fr\"{o}licher-Nijenhuis calculus, thus answering a question
posed by Huebschmann after a remark by Michor (see Remark 4.16 of \cite{h05}).

In Section \ref{CS}, I remark that the SH Lie-Rinehart algebra of a foliation
is independent of the complementary distribution appearing in the definition,
up to isomorphisms, and describe a canonical isomorphism between the SH
Lie-Rinehart algebras determined by different complementary distributions. In
Section \ref{SHLRPSF}, as a further example of the emergence of SH structures
in secondary calculus, I consider the integral foliation of the degeneracy
distribution of a presymplectic form and prove that there exists a canonical
morphism from the SH algebra of Oh and Park to the SH Lie-Rinehart algebra of
the foliation.

The third part of the paper contains the appendixes. The first
appendix contains some computational details omitted in Sections \ref{SHDGLR}
and \ref{MSHLR}. In the second appendix, I show that the higher brackets in a
SH Lie-Rinehart algebra are actually derived brackets, according to the
construction of T.{} Voronov \cite{v05}. In the third appendix, I briefly
present an alternative derivation of the SH Lie-Rinehart algebra of a
foliation which does not apply to the general case of a Lie subalgebroid.
Finally, in the last appendix, I present an alternative formulas for the
binary operations in the SH Lie-Rinehart algebra of a foliation which could be
useful for some purposes.

\subsection{Conventions and notations}

I will adopt the following notations and conventions throughout the paper. Let
$k_{1},\ldots,k_{\ell}$ be positive integers. I denote by $S_{k_{1}%
,\ldots,k_{\ell}}$ the set of $(k_{1},\ldots,k_{\ell})$\emph{-unshuffles},
i.e., permutations $\sigma$ of $\{1,\ldots,k_{1}+\cdots+k_{\ell}\}$ such that
\[
\sigma(k_{1}+\cdots+k_{i-1}+1)<\cdots<\sigma(k_{1}+\cdots+k_{i-1}+k_{i}),\quad
i=1,\ldots,\ell.
\]

If $S$ is a set, I denote
\[
S^{\times k}:={}\underset{k\text{ times}}{\underbrace{S\times\cdots\times S}%
},
\]
and the element $(s,\ldots,s)\in S^{\times k}$ of the diagonal will be simply
denoted by $s^{k}$, $s\in S$.

The degree of a homogeneous element $v$ in a graded vector space will be
denoted by $\bar{v}$. However, when it appears in the exponent of a sign
$(-)$, I will always omit the overbar, and write, for instance, $(-)^{v}$
instead of $(-)^{\bar{v}}$.

Every vector space will be over a field $K$ of zero characteristic, which will
actually be $\mathbb{R}$ in Part \ref{P2} (and Appendixes \ref{Ap3} and
\ref{Ap4}).

If $V=\bigoplus_{i}V^{i}$ is a graded vector space, I denote by
$V[1]=\bigoplus_{i}V[1]^{i}$ (resp., $V[-1]=\bigoplus_{i}V[-1]^{i}$) its
suspension (resp., de-suspension), i.e., the graded vector space defined by
putting $V[1]^{i}=V^{i+1}$ (resp., $V[-1]^{i}=V^{i-1}$). Let $V_{1}%
,\ldots,V_{n}$ be graded vector spaces,
\[
\boldsymbol{v}=(v_{1},\ldots,v_{n})\in V_{1}\times\cdots\times V_{n},
\]
and $\sigma$ a permutation of $\{1,\ldots,n\}$. I denote by $\alpha
(\sigma,\boldsymbol{v})$ (resp., $\chi(\sigma,\boldsymbol{v})$) the sign
implicitly defined by
\[
v_{\sigma(1)}\odot\cdots\odot v_{\sigma(n)}=\alpha(\sigma,\boldsymbol{v}%
)\,v_{1}\odot\cdots\odot v_{n}\quad\text{(resp., }v_{\sigma(1)}\wedge
\cdots\wedge v_{\sigma(n)}=\chi(\sigma,\boldsymbol{v})\,v_{1}\wedge
\cdots\wedge v_{n}\text{)}%
\]
where $\odot$ (resp., $\wedge$) is the graded symmetric (resp., graded
skew-symmetric) product in the symmetric (resp., exterior) algebra of
$V_{1}\oplus\cdots\oplus V_{n}$.

Let $V,W$ be graded vector spaces, $\Phi:V^{\times k}\longrightarrow W$ a
graded, multilinear map, $\boldsymbol{v}=(v_{1},\ldots,v_{k})\in V^{\times k}%
$, and $\sigma\in S_{k}$. I call $\alpha(\sigma,\boldsymbol{v})\Phi
(v_{\sigma(1)},\cdots,v_{\sigma(k)})$ a \emph{Koszul signed permutation} of
$(v_{1},\ldots,v_{k})$ (in $\Phi(v_{1},\ldots,v_{k})$).

Now, let $M$ be a smooth manifold. I denote by $C^{\infty}(M)$ the real
algebra of smooth functions on $M$, by $\mathfrak{X}(M)$ the Lie-Rinehart
algebra of vector fields on $M$, and by $\Lambda(M)$ the DG algebra of
differential forms on $M$. Elements in $\mathfrak{X}(M)$ are always understood
as derivations of $C^{\infty}(M)$. Homogeneous elements in $\Lambda(M)$ are
always understood as $C^{\infty}(M)$-valued, skew-symmetric, multilinear maps
on $\mathfrak{X}(M)$. I simply denote by $\omega_{1}\omega_{2}$ (instead of
$\omega_{1}\wedge\omega_{2}$) the (wedge) product of differential forms
$\omega_{1},\omega_{2}$. I denote by $d:\Lambda(M)\longrightarrow\Lambda(M)$
the exterior differential. Every tensor product will be over $C^{\infty}(M)$,
if not explicitly stated otherwise, and will be simply denoted by $\otimes$.
Finally, I adopt the Einstein summation convention.

\part{Algebraic Foundations}

\section{Strong Homotopy Structures\label{SHS}}

Let $(V,\delta)$ be a complex of vector spaces and $\mathscr{A}$ be any kind
of algebraic structure (associative algebra, Lie algebra, module, etc.).
Roughly speaking, a homotopy $\mathscr{A}$-structure on $(V,\delta)$ is an
algebraic structure on $V$ which is of the kind $\mathscr{A}$ only up to
$\delta$-homotopies, and a \emph{strong homotopy (SH)} $\mathscr{A}$%
\emph{-structure} is a homotopy structure possessing a full system of
(coherent) \emph{higher homotopies}. In this paper, I will basically deal with
three kinds of SH structures, namely SH {Lie} algebras (also named $L_{\infty
}$-algebras), SH {Lie} modules (also named $L_{\infty}$-modules), and SH
Lie-Rinehart algebras (that, actually, encompass the latter). For them I
provide detailed definitions below.

Let $L$ be a graded vector space, and let $\mathscr{L}=\{[{}\cdot{},\cdots
,{}\cdot{}]_{k},\ k\in\mathbb{N}\}$ be a family of $k$-ary, multilinear,
homogeneous of degree $2-k$ operations
\[
\lbrack{}\cdot{},\cdots,{}\cdot{}]_{k}:{}L^{\times k}\longrightarrow L,\quad
k\in\mathbb{N}.
\]
If the $[{}\cdot{},\cdots,{}\cdot{}]_{k}$'s are graded skew-symmetric, then
the $k$-\emph{th Jacobiator of }$\mathscr{L}$ is, by definition, the
multilinear, homogeneous of degree $3-k$ map\emph{ }%
\[
J^{k}:{}L^{\times k}\longrightarrow L,
\]
defined by
\[
J^{k}(v_{1},\ldots,v_{k}):=\sum_{i+j=k}(-)^{ij}\sum_{\sigma\in S_{i,j}}%
\chi(\sigma,\boldsymbol{v})\,[[v_{\sigma(1)},\ldots,v_{\sigma(i)}%
],v_{\sigma(i+1)},\ldots,v_{\sigma(i+j)}],
\]
$\boldsymbol{v}=(v_{1},\ldots,v_{k})\in L^{\times k}$.

I will often omit the subscript $k$ in $[{}\cdot{}{},\cdots,{}\cdot{}]_{k}$
when it is clear from the context, and I will do the same for other $k$-ary
operations in the paper without further comments.

\begin{definition}
\label{Def1}An $L_{\infty}$\emph{-algebra} is a pair $(L,\mathscr{L})$, where
$L$ is a graded vector space, and $\mathscr{L}=\{[{}\cdot{},\cdots,{}\cdot
{}]_{k},\ k\in\mathbb{N}\}$ is a family of $k$-ary, multilinear, homogeneous
of degree $2-k$ operations
\[
\lbrack{}\cdot{},\cdots,{}\cdot{}]_{k}:{}L^{\times k}\longrightarrow L,\quad
k\in\mathbb{N},
\]
such that

\begin{enumerate}
\item $[{}\cdot{},\cdots,{}\cdot{}]_{k}$ is graded skew-symmetric, and

\item the $k$-th Jacobiator of $\mathscr{L}$ vanishes identically,
\end{enumerate}

for all $k\in\mathbb{N}$, (in particular, $(L,[{}\cdot{}]_{1})$ is a complex).
\end{definition}

Notice that if $L$ is concentrated in degree $0$, then an $L_{\infty}$-algebra
structure on $L$ is simply a Lie algebra structure for degree reasons.
Similarly, if $[{}\cdot{},\cdots,{}\cdot{}]_{k}=0$ for all $k>2$, then
$(L,\mathscr{L})$ is a DG Lie algebra.

Now, let $(L,\mathscr{L})$, $\mathscr{L}=\{[{}\cdot{},\cdots,{}\cdot{}%
]_{k},\ k\in\mathbb{N}\}$, be an $L_{\infty}$-algebra, $M$ a graded vector
space, and let $\mathscr{M}=\{[{}\cdot{},\cdots,{}\cdot|{}\cdot{}{}%
]_{k},\ k\in\mathbb{N}\}$ be a family of $k$-ary, multilinear, homogeneous of
degree $2-k$ operations,
\[
\lbrack{}\cdot{},\cdots,{}\cdot|{}\cdot{}{}]_{k}:{}L^{\times(k-1)}\times
M\longrightarrow M,\quad k\in\mathbb{N}.
\]
If the $[{}\cdot{},\cdots,{}\cdot|{}\cdot{}{}]_{k}$'s are graded
skew-symmetric in the first $k-1$ entries, then the $k$\emph{-th Jacobiator
of }$\mathscr{M}$ is, by definition, the multilinear, homogeneous of degree
$3-k$, map
\[
J^{k}:{}L^{\times(k-1)}\times M\longrightarrow M,
\]
defined by
\[
J^{k}(v_{1},\ldots,v_{k-1}|m):=\sum_{i+j=k}(-)^{ij}\sum_{\sigma\in S_{i,j}%
}\chi(\sigma,\boldsymbol{b})\,[[b_{\sigma(1)},\ldots,b_{\sigma(i)}]^{\oplus
},b_{\sigma(i+1)},\ldots,b_{\sigma(i+j)}]^{\oplus},
\]
$\boldsymbol{b}=(v_{1},\ldots,v_{k-1},m)\in L^{\times(k-1)}\times M$, where
the $[{}\cdot{}{},\cdots,{}\cdot{}]^{\oplus}$'s are new operations
\[
\lbrack{}\cdot{}{},\cdots,{}\cdot{}]_{k}^{\oplus}:{}(L\oplus M)^{\times
k}\longrightarrow L\oplus M,\quad k\in\mathbb{N},
\]
defined by extending the $[{}\cdot{},\cdots,{}\cdot{}]_{k}$'s and the
$[{}\cdot{},\cdots,{}\cdot|{}\cdot{}{}]_{k}$'s by multilinearity,
skew-symmetry, and the condition that the result is zero if more than one
entry is from $M$.

\begin{definition}
\label{Def2}An $L_{\infty}$\emph{-module} over the $L_{\infty}$-algebra
$(L,\mathscr{L})$, $\mathscr{L}=\{[{}\cdot{},\cdots,{}\cdot{}]_{k}%
,\ k\in\mathbb{N}\}$, is a pair $(M,\mathscr{M})$, where $M$ is a graded
vector space, and $\mathscr{M}=\{[{}\cdot{},\cdots,{}\cdot|{}\cdot{}{}%
]_{k},\ k\in\mathbb{N}\}$ is a family of $k$-ary, multilinear, homogeneous of
degree $2-k$ operations,
\[
\lbrack{}\cdot{},\cdots,{}\cdot|{}\cdot{}{}]_{k}:{}L^{\times(k-1)}\times
M\longrightarrow M,\quad k\in\mathbb{N},
\]
such that

\begin{enumerate}
\item $[{}\cdot{},\cdots,{}\cdot|{}\cdot{}{}]_{k}$ is graded skew-symmetric in
the first $k-1$ entries, and

\item the $k$-th Jacobiator of $\mathscr{M}$ vanishes identically,

for all $k\in\mathbb{N}$ (in particular, $(M,[{}\cdot{}]_{1})$ is a complex).
\end{enumerate}
\end{definition}

If both $L$ and $M$ are concentrated in degree $0$, then an $L_{\infty}%
$-module structure on $M$ over $L$ is simply a Lie module structure over the
Lie algebra $L$. Similarly, if $[{}\cdot{},\cdots,{}\cdot{}]_{k}=0$ and
$[{}\cdot{},\cdots,{}\cdot{}|{}\cdot{}]_{k}=0$ for all $k>2$, then
$(M,\mathscr{M})$ is a DG Lie module over the DG Lie algebra $L$.

The sign conventions in Definitions \ref{Def1} and \ref{Def2} are the same as
in \cite{ls93,lm95}. However, in this paper, I will mainly use a different
sign convention \cite{v05}. Namely, I will deal with what are often called
$L_{\infty}[1]$\emph{-algebras} and $L_{\infty}[1]$\emph{-modules}, whose
definitions I recall now.

Let $L$ be a graded vector space, and let $\mathscr{L}=\{\{{}\cdot{}{}%
,\cdots,{}\cdot{}\}_{k},\ k\in\mathbb{N}\}$ be a family of $k$-ary,
multilinear, homogeneous of degree $1$ operations
\[
\{{}\cdot{}{},\cdots,{}\cdot{}\}_{k}:{}L^{\times k}\longrightarrow L,\quad
k\in\mathbb{N}.
\]
If the $\{{}\cdot{}{},\cdots,{}\cdot{}\}_{k}$'s are graded symmetric, then the
$k$-\emph{th Jacobiator of }$\mathscr{L}$ is, by definition, the multilinear,
homogeneous of degree $2$ map\emph{ }%
\[
J^{k}:{}L^{\times k}\longrightarrow L,
\]
defined by
\[
J^{k}(v_{1},\ldots,v_{k}):=\sum_{i+j=k}\sum_{\sigma\in S_{i,j}}\alpha
(\sigma,\boldsymbol{v})\,\{\{v_{\sigma(1)},\ldots,v_{\sigma(i)}\},v_{\sigma
(i+1)},\ldots,v_{\sigma(i+j)}\},
\]
$\boldsymbol{v}=(v_{1},\ldots,v_{k})\in L^{\times k}$.

\begin{definition}
An $L_{\infty}[1]$\emph{-algebra} is a pair $(L,\mathscr{L})$, where $L$ is a
graded vector space, and $\mathscr{L}=\{\{{}\cdot{}{},\cdots,{}\cdot{}%
\}_{k},\ k\in\mathbb{N}\}$ is a family of $k$-ary, multilinear, homogeneous of
degree $1$ operations,
\[
\{{}\cdot{}{},\cdots,{}\cdot{}\}_{k}:{}L^{k}\longrightarrow L,\quad
k\in\mathbb{N},
\]
such that

\begin{enumerate}
\item $\{{}\cdot{}{},\cdots,{}\cdot{}\}_{k}$ is graded symmetric,

\item the $k$-th Jacobiator of $\mathscr{L}$ vanishes identically,
\end{enumerate}

for all $k\in\mathbb{N}$.
\end{definition}

There is a one-to-one correspondence between $L_{\infty}$-algebra structures
$\{[{}\cdot{}{},\cdots,{}\cdot{}]_{k},\ k\in\mathbb{N}\}$ in a graded vector
space $L$, and $L_{\infty}[1]$-algebra structures $\{\{{}\cdot{}{},\cdots
,{}\cdot{}\}_{k},\ k\in\mathbb{N}\}$ in $L[1]$, given by
\[
\{{}v_{1}{},\ldots,{}v_{k}\}=(-)^{(k-1)v_{1}+(k-2)v_{2}+\cdots+v_{k-1}}%
[{}v_{1}{},\cdots,{}v_{k}],\quad v_{1},\ldots,v_{k}\in L,\quad k\in
\mathbb{N}.
\]

$L_{\infty}[1]$-algebras build up a category whose morphisms are defined as follows.

Let $(L,\{\{{}\cdot{}{},\cdots,{}\cdot{}\}_{k},\ k\in\mathbb{N}\})$ and
$(L^{\prime},\{\{{}\cdot{}{},\cdots,{}\cdot{}\}_{k}^{\prime},\ k\in
\mathbb{N}\})$ be $L_{\infty}[1]$-algebras, and let $f=\{f_{k},\ k\in
\mathbb{N}\}$ be a family of $k$-ary, multilinear, homogeneous of degree $0$
maps%
\[
f_{k}:{}L^{\times k}\longrightarrow L^{\prime},\quad k\in\mathbb{N}.
\]
If the $f_{k}$'s are graded symmetric, define multilinear, homogeneous of
degree $1$ maps
\[
K_{f}^{k}:{}L^{\times k}\longrightarrow L^{\prime},
\]
by putting
\begin{align*}
&  K_{f}^{k}(v_{1},\ldots,v_{k})\\
&  :=\sum_{i+j=k}\sum_{\sigma\in S_{i,j}}\alpha(\sigma,\boldsymbol{v}%
)\,f_{i+j+1}(\{v_{\sigma(1)},\ldots,v_{\sigma(i)}\},v_{\sigma(i+1)}%
,\ldots,v_{\sigma(i+j)})\\
&  \quad\ -\sum_{\ell=1}^{k}\sum_{\substack{k_{1}+\cdots+k_{\ell}=k\\k_{1}%
\leq\cdots\leq k_{\ell}}}\sum_{\sigma\in S_{k_{1},\ldots,k_{\ell}}^{<}}%
\alpha(\sigma,\boldsymbol{v})\{f_{k_{1}}(v_{\sigma(1)},\ldots,v_{\sigma
(k_{1})}),\ldots,f_{k_{\ell}}(v_{\sigma(k-k_{\ell}+1)},\ldots,v_{\sigma
(k)})\}^{\prime},
\end{align*}
$\boldsymbol{v}=(v_{1},\ldots,v_{k})\in L^{\times k}$, where $S_{k_{1}%
,\ldots,k_{\ell}}^{<}\subset S_{k_{1},\ldots,k_{\ell}}$ is the set of
$(k_{1},\ldots,k_{\ell})$-unshuffles such that
\[
\sigma(k_{1}+\cdots+k_{i-1}+1)<\sigma(k_{1}+\cdots+k_{i-1}+k_{i}+1)\text{\quad
whenever }k_{i}=k_{i+1}.
\]

\begin{definition}
A \emph{morphism} $f:L\longrightarrow L^{\prime}$ of the $L_{\infty}%
[1]$-algebras $(L,\{\{{}\cdot{}{},\cdots,{}\cdot{}\}_{k},\ k\in\mathbb{N}\})$
and $(L^{\prime},\{\{{}\cdot{}{},\cdots,{}\cdot{}\}_{k}^{\prime}%
,\ k\in\mathbb{N}\})$ is a family $f=\{f_{k},\ k\in\mathbb{N}\}$ of $k$-ary,
multilinear, homogeneous of degree $0$ maps,
\[
f_{k}:{}L^{\times k}\longrightarrow L^{\prime},\quad k\in\mathbb{N},
\]
such that

\begin{enumerate}
\item $f_{k}$ is graded symmetric, and

\item $K_{f}^{k}$ vanishes identically,
\end{enumerate}

\noindent for all $k\in\mathbb{N}$.
\end{definition}

An \emph{identity morphism} $\mathbb{I}:L\longrightarrow L$ is defined by
$\mathbb{I}:=\{\mathbb{I}_{k},\ k\in\mathbb{N}\}$, where $\mathbb{I}%
_{1}:L\longrightarrow L$ is the identity map, and $\mathbb{I}_{k}:L^{\times
k}\longrightarrow L$ is the zero map for $k>1\mathrm{.}$

If $f:L\longrightarrow L^{\prime}$ and $g:L^{\prime}\longrightarrow
L^{\prime\prime}$ are morphisms of $L_{\infty}[1]$-algebras, the composition
$g\circ f:L\longrightarrow L^{\prime\prime}$ is defined as $g\circ
f:=\{(g\circ f)_{k},\ k\in\mathbb{N}\}$, where
\begin{align*}
&  (g\circ f)_{k}(v_{1},\ldots,v_{k})\\
&  :=\sum_{\ell=1}^{k}\sum_{\substack{k_{1}+\cdots+k_{\ell}=k\\k_{1}\leq
\cdots\leq k_{\ell}}}\sum_{\sigma\in S_{k_{1},\ldots,k_{\ell}}^{<}}%
\alpha(\sigma,\boldsymbol{v})g_{\ell}(f_{k_{1}}(v_{\sigma(1)},\ldots
,v_{\sigma(k_{1})}),\ldots,f_{k_{\ell}}(v_{\sigma(k-k_{\ell}+1)}%
,\ldots,v_{\sigma(k)}))
\end{align*}
for all $\boldsymbol{v}=(v_{1},\ldots,v_{k})\in L^{\times k}$, $k\in
\mathbb{N}$. $g\circ f$ is a morphism as well.

Now, let $(L,\mathscr{L})$, $\mathscr{L}=\{\{{}\cdot{},\cdots,{}\cdot{}%
\}_{k},\ k\in\mathbb{N}\}$, be an $L_{\infty}[1]$-algebra, $M$ a graded vector
space, and let $\mathscr{M}=\{\{{}\cdot{},\cdots,{}\cdot|{}\cdot{}{}%
\}_{k},\ k\in\mathbb{N}\}$ be a family of $k$-ary, multilinear, homogeneous of
degree $1$ operations,
\[
\{{}\cdot{},\cdots,{}\cdot|{}\cdot{}{}\}_{k}:{}L^{\times(k-1)}\times
M\longrightarrow M,\quad k\in\mathbb{N}.
\]
If the $\{{}\cdot{},\cdots,{}\cdot|{}\cdot{}{}\}_{k}$'s are graded symmetric
in the first $k-1$ entries, then the $k$\emph{-th Jacobiator of
}$\mathscr{M}$ is, by definition, the multilinear, homogeneous of degree $2$,
map\emph{ }%
\[
J^{k}:{}L^{\times(k-1)}\times M\longrightarrow M,
\]
defined by
\begin{equation}
J^{k}(v_{1},\ldots,v_{k-1}|m):=\sum_{i+j=k}\sum_{\sigma\in S_{i,j}}%
\alpha(\sigma,\boldsymbol{b})\,\{\{b_{\sigma(1)},\ldots,b_{\sigma
(i)}\}^{\oplus},b_{\sigma(i+1)},\ldots,b_{\sigma(i+j)}\}^{\oplus},
\label{Jac2}%
\end{equation}
$\boldsymbol{b}=(v_{1},\ldots,v_{k-1},m)\in L^{\times(k-1)}\times M$, where
the $\{{}\cdot{}{},\cdots,{}\cdot{}\}^{\oplus}$'s are new operations
\[
\{{}\cdot{}{},\cdots,{}\cdot{}\}_{k}^{\oplus}:{}(L\oplus M)^{\times
k}\longrightarrow L\oplus M,\quad k\in\mathbb{N},
\]
defined by extending the $\{{}\cdot{},\cdots,{}\cdot{}\}_{k}$'s and the
$\{{}\cdot{},\cdots,{}\cdot|{}\cdot{}{}\}_{k}$'s by multilinearity, symmetry,
and the condition that the result is zero if more than one entry is from $M$.

\begin{definition}
An $L_{\infty}[1]$\emph{-module} over the $L_{\infty}[1]$-algebra
$(L,\mathscr{L})$, $\mathscr{L}=\{\{{}\cdot{},\cdots,{}\cdot{}\}_{k}%
,\ k\in\mathbb{N}\}$, is a pair $(M,\mathscr{M})$, where $M$ is a graded
vector space, and $\mathscr{M}=\{\{{}\cdot{},\cdots,{}\cdot|{}\cdot{}{}%
\}_{k},\ k\in\mathbb{N}\}$ is a family of $k$-ary, multilinear, homogeneous of
degree $1$, operations,
\[
\{{}\cdot{},\cdots,{}\cdot|{}\cdot{}{}\}_{k}:{}L^{\times(k-1)}\times
M\longrightarrow M,\quad k\in\mathbb{N},
\]
such that

\begin{enumerate}
\item $\{{}\cdot{},\cdots,{}\cdot|{}\cdot{}{}\}_{k}$ is graded symmetric in
the first $k-1$ entries, and

\item the $k$-th Jacobiator of $\mathscr{M}$ vanishes identically, for all
$k\in\mathbb{N}$.
\end{enumerate}
\end{definition}

Finally, define SH Lie-Rinehart algebras. I will use the same sign convention
as in the definition of $L_{\infty}[1]$-algebras (and $L_{\infty}%
[1]$-modules). For simplicity, I call the resulting objects $LR_{\infty}%
[1]$\emph{-algebras}. To my knowledge, (a version of) this definition has been
proposed for the first time {by Kjeseth} in \cite{k01}. Recall that a
Lie-Rinehart algebra is a (purely algebraic) generalization of a Lie algebroid.

\begin{definition}
\label{defLRalg}An $LR_{\infty}[1]$\emph{-algebra} is a pair $(A,\mathcal{Q}%
)$, where $A$ is an associative, graded commutative, unital algebra, and
$(\mathcal{Q},\mathscr{Q})$ is an $L_{\infty}[1]$-algebra, $\mathscr{Q}=\{\{{}%
\cdot{}{},\cdots,{}\cdot{}\}_{k},\ k\in\mathbb{N}\}$. Moreover,

\begin{itemize}
\item {$\mathcal{Q}$} possesses the structure of an{ $A$-module,}

\item $A$ possesses the structure $\mathscr{M}=\{\{{}\cdot{}{},\cdots,{}%
\cdot{}|{}\cdot{}\}_{k},\ k\in\mathbb{N}\}$ of an $L_{\infty}[1]$-module over
$(\mathcal{Q},\mathscr{Q})$,
\end{itemize}

such that,

\begin{itemize}
\item[--] $\{{}\cdot{},\cdots,{}\cdot{}|{}\cdot{}\}_{k}:\mathcal{Q}%
^{\times(k-1)}\times A\longrightarrow A$ is a derivation in the last entry
and is $A$-multilinear in the first $k-1$ entries;

\item[--] Formula
\begin{equation}
\{q_{1},\ldots,q_{k-1},aq_{k}\}=\{q_{1},\ldots,q_{k-1}\,|\,a\}{}\cdot
q_{k}+(-)^{a\boldsymbol{(}q_{1}+\cdots+q_{k-1}+1)}a\cdot\{q_{1},\ldots
,q_{k-1},q_{k}\}, \label{LRP}%
\end{equation}
holds for all $q_{1},\ldots,q_{k}\in\mathcal{Q}$, $a\in A$, $k\in\mathbb{N}$
(in particular, $(\mathcal{Q},\{{}\cdot{}\}_{1})$ is a DG module over
$(A,\{{}|{}\cdot{}{}\}_{1})$).
\end{itemize}

The map $\{{}\cdot{},\cdots,{}\cdot{}|{}\cdot{}\}_{k}:\mathcal{Q}%
^{\times(k-1)}\times A\longrightarrow A$ is called the $k$\emph{th anchor,
}$k\in\mathbb{N}$.
\end{definition}

Note that the brackets $\{{}\cdot{},\cdots,{}\cdot{}\}_{k}$ in the above
definition are only $K$-linear, in general. Formula (\ref{LRP}) is a higher
generalization of the standard identity fulfilled by the anchor in a
Lie-Rinehart algebra.

If $\mathcal{Q}$ is concentrated in degree $-1$, and $A$ is concetrated in
degree $0$, $(A,\mathcal{Q}[-1])$ is a Lie-Rinehart algebra.

{In the \emph{smooth setting}, i.e., when $A$ is the algebra of smooth
functions on a smooth manifold }$M${ (in particular }$A$ is concentrated in
degree $0${), and $\mathcal{Q}[-1]$ is the }$A$-module of sections of a graded
bundle $\mathcal{E}$ over {$\mathcal{M}$}, then $\mathcal{E}${ is sometimes
called an $L_{\infty}$-algebroid \cite{sss09,sz11,b11,bp12}.}

\begin{remark}
\label{Hue1}In \cite{h05}, Huebschmann proposes a definition of a homotopy
version of a Lie-Rinehart algebra, called a \emph{quasi Lie-Rinehart algebra}.
{Although he mentions} the earlier work \cite{k01} of Kjeseth, he doesn't
discuss the relation between quasi Lie-Rinehart algebras and Kjeseth's
homotopy Lie-Rinehart pairs. For instance, he doesn't state explicitly that a
quasi Lie-Rinehart algebra is, in particular, an $L_{\infty}$-algebra.
Actually, this is an immediate consequence of the description of $LR_{\infty
}[1]$-algebras in terms of their Chevalley-Eilenberg algebras (also known as
Maurer-Cartan algebras) discussed in the next section.
\end{remark}

\section{Homotopy Lie-Rinehart Algebras and Multi-Differential
Algebras\label{SHDGLR}}

In this section, I discuss a DG algebraic approach to $LR_{\infty}%
[1]$-algebras, which is especially suited for the aim of this paper, where the
main $LR_{\infty}[1]$-algebra comes from a DG algebra of differential forms.
Propositions in this section are known to specialists but, to my knowledge,
explicit formulas and detailed proofs are not available. I include some of
them here and others in Appendix \ref{Appendix}.

Notice that the approach in terms of DG algebras (as opposed to the one in
terms of coalgebras) has the slight disadvantage of necessitating extra
finiteness conditions: a certain module has to possess a nice biduality property.

{Let $A$ be an associative, commutative, unital algebra over a field $K$ of
zero characteristic, and $Q$ an $A$-module. It is well known that a
Lie-Rinehart algebra structure on $(A,Q)$ determines a homological derivation
$D$ in the graded algebra $\mathrm{Alt}_{A} (Q,A)$ of alternating, $A$-valued,
$A$-multilinear maps on $Q$. The DG algebra $(\mathrm{Alt}_{A} (Q,A),D)$ is
the \emph{Chevalley-Eilenberg algebra} of $Q$. On the other hand, if $Q$ is projective and finitely
generated, then $\mathrm{Alt}_{A} (Q,A)$ is isomorphic to $\Lambda^{\bullet
}_{A}Q^{\ast}$, the exterior algebra of the dual module, and a homological
derivation in it determines a Lie-Rinehart algebra structure on $(A,Q)$. }

{Similarly, let $A$ be a commutative, unital $K$-algebra, and
$\mathcal{Q}$ a graded $A$-module. An $LR_{\infty}[1]$-algebra structure on
$(A,\mathcal{Q}$ determines a formal homological derivation $D$
in the graded algebra $\mathrm{Sym}_{A}(\mathcal{Q},A)$ of graded, graded
symmetric, $A$-valued, $A$-multilinear maps on $\mathcal{Q}$ (see below). In
\cite{k01}, Kjeseth describes $D$ in \emph{coalgebraic terms} and call
$(\mathrm{Sym}_{A}(\mathcal{Q},A),D)$ the \emph{homotopy Rinehart complex} of
$\mathcal{Q}$. On the other hand, if $\mathcal{Q}$ is projective and finitely
generated, then $\mathrm{Sym}_{A}(\mathcal{Q},A)\simeq S_{A}^{\bullet
}\mathcal{Q}^{\ast}$, the graded symmetric algebra of the dual module, and a
formal homological derivation in it determines an $LR_{\infty}[1]$-algebra
structure on $\mathcal{Q}$. This is shown below. Instead of using the language
of formal derivations, I prefer to use the language of multi-differential
algebra structures (named multi-algebras in \cite{h05}), which makes manifest
the role of \emph{higher homotopies}.}

\begin{definition}
A \emph{multi-differential algebra} is a pair $(\mathcal{A},\mathscr{D})$,
where $\mathcal{A}=\bigoplus_{r,s}\mathcal{A}^{r,s}$ is a bi-graded algebra,
understood as a graded algebra with respect to the \emph{total degree} $r+s$
(now on, named simply the \emph{degree}), and $\mathscr{D}=\{d_{k}%
,\,k\in\mathbb{N}_{0}\}$ is a family of graded derivations
\[
d_{k}:\mathcal{A}\longrightarrow\mathcal{A},
\]
of bi-degree $(k,-k+1)$ (in particular $d_{k}$ is homogeneous of degree $1$),
such that the derivations
\[
E_{k}:=\sum_{i+j=k}[d_{i},d_{j}]:{}\mathcal{A}\longrightarrow\mathcal{A},\quad
k\in\mathbb{N}
\]
vanish for all $k\in\mathbb{N}$ (in particular, $(\mathcal{A},d_{0})$ is a DG algebra).
\end{definition}

Huebschmann \cite{h05} introduces multi-differential algebras, under the name
of \emph{multi-algebras}, but then he concentrates on the case $d_{k}=0$ for
$k>2$. Indeed, a multi-differential algebra with $d_{k}=0$ for $k>2$ is
naturally associated with a Lie subalgebroid in a Lie algebroid. However, the
general case is relevant as well. For instance, multi-differential algebras
are at the basis of the BFV-BRST formalism \cite{ht92,s97} (see also
\cite{k01b}).

\begin{remark}
\ {If for any homogeneous element $\omega\in\mathcal{A}$, $d_{k}\omega=0$ for
$k\gg1$, one can consider the derivation $D:=\sum_{k}d_{k}$ (otherwise $D$ is
just a formal derivation)}. Condition $E_{k}=0$ for all $k$ is then equivalent
to $D^{2}=0$.
\end{remark}

{Now, let $A$ be an associative, graded commutative, unital algebra, and
$\mathcal{Q}$ a graded $A$-module. Let $\mathrm{Sym}_{A}^{r}(\mathcal{Q},A)$
be the graded $A$-module of graded, graded symmetric, $A$-multilinear maps
with $r$ entries. A homogeneous element $\omega\in\mathrm{Sym}_{A}%
^{r}(\mathcal{Q},A)$ is a homogeneous, graded symmetric, $K$-multilinear map
\[
\omega:\mathcal{Q}^{r}\longrightarrow A,
\]
such that
\[
\omega(aq_{1},q_{2},\ldots,q_{r})=(-)^{a\omega}a\omega(q_{1},q_{2}%
,\ldots,q_{r}),\quad a\in A,\quad q_{1},\ldots,q_{r}\in\mathcal{Q}.
\]
In particular, $\mathrm{Sym}_{A}^{0}(\mathcal{Q},A)=A$ and $\mathrm{Sym}%
_{A}^{1}(\mathcal{Q},A)=\mathrm{Hom}_{A}(\mathcal{Q},A)$ (I will also denote
it by $\mathcal{Q}^{\ast}$). Consider also $\mathrm{Sym}_{A}(\mathcal{Q}%
,A)=\bigoplus_{r}\mathrm{Sym}_{A}^{r}(\mathcal{Q},A)$. It is a bi-graded $A$-module in an obvious way.
Moreover, $\mathrm{Sym}_{A}(\mathcal{Q},A)$ is a graded, associative, graded
commutative, unital algebra: for $\omega\in\mathrm{Sym}_{A}^{r}(\mathcal{Q}%
,A)$, $\omega^{\prime}\in\mathrm{Sym}_{A}^{r^{\prime}}(\mathcal{Q},A)$,
$q_{1},\ldots,q_{r+r^{\prime}}\in\mathcal{Q}$,
\begin{equation}
(\omega\omega^{\prime})(q_{1},\ldots,q_{r+r^{\prime}})=\sum_{\sigma\in
S_{r,r^{\prime}}}(-)^{\omega^{\prime}(q_{\sigma(1)}+\cdots+q_{\sigma(r)}%
)}\alpha(\sigma,\boldsymbol{q})\omega(q_{\sigma(1)},\ldots,q_{\sigma
(r)})\omega^{\prime}(q_{\sigma(r+1)},\ldots q_{\sigma(r+r^{\prime})}).
\label{Prod}%
\end{equation}
Since $\omega\omega^{\prime}\in\mathrm{Sym}_{A}^{r+r^{\prime}}(\mathcal{Q}%
,A)$, } the algebra {$\mathrm{Sym}_{A}(\mathcal{Q},A)$} is actually bigraded.
Namely, if $\omega\in{\mathrm{Sym}_{A}^{r}(\mathcal{Q},A)}$ is homogeneous,
then its bidegree is defined as $(r,\bar{\omega}-r)$. I denote by
{$\mathrm{Sym}_{A}^{r}(\mathcal{Q},A)^{s}$} the subspace of elements in
{$\mathrm{Sym}_{A}(\mathcal{Q},A)$} of bidegree $(r,s)$.

\begin{remark}
\label{15}Suppose $\mathcal{Q}=A\otimes_{A_{0}}Q[1]$, where $A_{0}$ is the
zeroth homogenous component of $A$ and $Q$ is a projective and finitely
generated $A_{0}$-module. Then {$\mathrm{Sym}_{A}(\mathcal{Q},A$}${)}\simeq
A\otimes_{A_{0}}\Lambda_{A_{0}}^{\bullet}Q^{\ast}$ as an $A$-module. There is
a pre-existing, graded algebra structure on $A\otimes_{A_{0}}\Lambda_{A_{0}%
}^{\bullet}Q^{\ast}$ given by the exterior product
\[
(a\otimes\omega)\wedge(b\otimes\xi):=(-)^{b\omega}ab\otimes\omega\wedge
\xi,\quad a,b\in A,\quad\omega,\xi\in Q^{\ast}.
\]
The isomorphism {$\mathrm{Sym}_{A}(\mathcal{Q},A$}${)}\simeq A\otimes_{A_{0}%
}\Lambda_{A_{0}}^{\bullet}Q^{\ast}$ can be chosen so that it identifies the
algebra structures in $A\otimes_{A_{0}}\Lambda_{A_{0}}^{\bullet}Q^{\ast}$ and
{$\mathrm{Sym}_{A}(\mathcal{Q},A)$}. In order to do that, one should identify
$a\otimes\omega\in A\otimes_{A_{0}}\Lambda_{A_{0}}^{r}Q^{\ast}$ with the
unique element $\Omega$ in {$\mathrm{Sym}_{A}^{r}(\mathcal{Q},A)$} such that
\[
\Omega(q_{1},\ldots,q_{r})=(-)^{r(r-1)/2}a\omega(q_{1},\ldots,q_{r})
\]
for all $q_{1},\ldots,q_{r}\in Q$.
\end{remark}

\begin{theorem}
\label{23}{Let $A$ be an associative, graded commutative, unital algebra and
$\mathcal{Q}$ a projective and finitely generated $A$-module.} An $LR_{\infty
}[1]$-algebra structure on $(A,\mathcal{Q})$ is equivalent to a
multi-differential algebra structure $\{d_{k},\,k\in\mathbb{N}_{0}\}$ on
{$\mathrm{Sym}_{A}(\mathcal{Q},A$}${)}$.
\end{theorem}

\begin{remark}
\label{Hue2}Huebschmann \cite{h05} basically defines a \emph{quasi
Lie-Rinehart algebra} as the datum of an $A$-module $\mathcal{Q}$ and a
multi-differential algebra structure $\{d_{k},\,k\in\mathbb{N}_{0}\}$ on
{$\mathrm{Sym}_{A}(\mathcal{Q},A)$} such that $d_{k}=0$ for $k>2$. The
derivations $d_{0},d_{1},d_{2}$ induce unary, binary and tertiary brackets in
$\mathcal{Q}$, and unary and binary anchors. However, in \cite{h05}
Huebschmann does not spell out explicitly all identities determined, among the
former operations, by the identities among the $d_{k}$'s. The proof of Theorem
\ref{23} fills this out. In particular, it shows explicitly that a quasi
Lie-Rinehart algebra is a special case of an $LR_{\infty}[1]$-algebra, thus
relating definitions by Huebschmann and Kjeseth. Notice also that, very
recently, Huebschmann himself \cite{h13} proposed an alternative proof of
Theorem \ref{23} using the language of cocommutative coalgebras and twisting cochains.
\end{remark}

\begin{remark}
It is well known that the datum of a Lie-Rinehart algebra structure on a
module $Q$ is also equivalent to the datum of a \emph{suitable} Poisson
(resp., Schouten) algebra structure on $S^{\bullet}Q$ (resp., $\Lambda
^{\bullet}Q$) (see, for instance, \cite{k04}, for details). Similarly, the
datum of an $LR_{\infty}[1]$-algebra structure on $\mathcal{Q}$ (over a
commutative DG algebra $A$) is also equivalent to the datum of a
\emph{suitable} homotopy, Schouten (resp., Poisson) algebra structure on
$\Lambda_{A}^{\bullet}\mathcal{Q}[-1]$ (resp., $S_{A}^{\bullet}\mathcal{Q}%
[-1]$) (see \cite{b11} for details). The description of $LR_{\infty}%
[1]$-algebra structures in terms of multi-differential algebras, however,
looks the most convenient for the purposes of this paper (see Section
\ref{SHLRF}).
\end{remark}

\begin{proof}
[Proof of Theorem \ref{23}]Here is a sketch. I postpone the (computational)
details to Appendix \ref{Appendix}. Let $(A,\mathcal{Q})$ possess the
structure of an $LR_{\infty}[1]$-algebra. Denote brackets and anchors as
usual. Since $\mathcal{Q}$ is projective and finitely generated,
{$\mathrm{Sym}_{A}(\mathcal{Q},$}${A)}\simeq S_{A}^{\bullet}\mathcal{Q}^{\ast}$, which
is generated by $A$ and $\mathcal{Q}^{\ast}$. Define $d_{k}:{\mathrm{Sym}%
_{A}(\mathcal{Q},A)}\longrightarrow{\mathrm{Sym}_{A}(\mathcal{Q},A)}$ on
generators
in the following way. For $a\in A$, put
\[
(d_{k}a)(q_{1},\ldots,q_{k}):=(-)^{a(q_{1}+\cdots+q_{k})}\{q_{1},\ldots
,q_{k}|a\},\quad q_{1},\ldots,q_{k}\in\mathcal{Q},\quad k\geq0.
\]
Obviously $d_{k}a\in{\mathrm{Sym}_{A}^{k}(\mathcal{Q},A)}$. Similarly, for
$\omega\in\mathcal{Q}^{\ast}$, put
\[
\ (d_{k}\omega)(q_{1},\ldots,q_{k+1}):=\sum_{i=1}^{k+1}(-)^{\chi}%
\{q_{1},\ldots,\widehat{q_{i}},\ldots,q_{k+1}|\omega(q_{i})\}+(-)^{\omega
}\omega(\{q_{1},\ldots,q_{k+1}\}),
\]
where $\chi:=\bar{\omega}(\bar{q}_{1}+\cdots+\widehat{\bar{q}_{i}}+\cdots
\bar{q}_{k+1})+\bar{q}_{i}(\bar{q}_{i+1}+\cdots+\bar{q}_{k+1})$, $q_{1}%
,\ldots,q_{k+1}\in\mathcal{Q}$, a hat $\widehat{\cdot}$ denotes omission, and
$k\geq0$. It is easy to show that $d_{k}\omega\in{\mathrm{Sym}_{A}%
^{k+1}(\mathcal{Q},A)}$, in particular, it is
\textcolor{green}{$A$-}multilinear (see Lemma \ref{Alemma1} in Appendix
\ref{Appendix}).

Now
define $d_{k}$ by (uniquely) extending to ${\mathrm{Sym}_{A}(\mathcal{Q},A)}$
as a derivation. This is possible, indeed, for $a,\omega$ as above,
\[
d_{k}(a\omega)=(d_{k}a)\omega+(-)^{a}a(d_{k}\omega),
\]
(see Lemma \ref{Alemma2} in Appendix \ref{Appendix}). Notice that $d_{k}$
satisfies the following \emph{higher Chevalley-Eilenberg formula}:%
\begin{align*}
&  (d_{k}\omega)(q_{1},\ldots,q_{r+k})\\
&  :=\sum_{\sigma\in S_{k,r}}(-)^{\omega(q_{\sigma(1)}+\cdots+q_{\sigma(k)}%
)}\alpha(\sigma,\boldsymbol{q})\{q_{\sigma(1)},\ldots,q_{\sigma(k)}%
\,|\,\omega(q_{\sigma(k)},\ldots,q_{\sigma(k+r)})\}\\
&  \quad\ -\sum_{\tau\in S_{k+1,r-1}}(-)^{\omega}\alpha(\tau,\boldsymbol{q}%
)\omega(\{q_{\tau(1)},\ldots,q_{\tau(k+1)}\},q_{\tau(k+1)},\ldots
,q_{\tau(k+r)}),
\end{align*}
$\omega\in{\mathrm{Sym}_{A}^{r}(\mathcal{Q},A)}$, $q_{1},\ldots,q_{r+k}%
\in\mathcal{Q}$ (see Lemma \ref{Alemma3} in Appendix \ref{Appendix}).

It remains to prove that $E_{k}:=\sum_{\ell+m=k}[d_{\ell},d_{m}]=0$. It is a
degree $2$ derivation of ${\mathrm{Sym}_{A}(\mathcal{Q},A)}$. To show that it
vanishes, it is enough to prove that it vanishes on $A$ and $\mathcal{Q}%
^{\ast}$. Now, for $a\in A$, $\omega\in\mathcal{Q}^{\ast}$ and $q_{1}%
,\ldots,q_{k}\in\mathcal{Q}$
\[
(E_{k}a)(q_{1},\ldots,q_{k})=(-)^{a(q_{1}+\cdots+q_{k})}J^{k+1}(q_{1}%
,\ldots,q_{k}\,|\,a)=0,
\]
and
\[
(E_{k}\omega)(q_{1},\ldots,q_{k+1})=\sum_{i=1}^{k+1}(-)^{\chi}J^{k+1}%
(q_{1},\ldots,\widehat{q_{i}},\ldots,q_{k+1}\,|\,\omega(q_{i}))-\omega
(J^{k+1}(q_{1},\ldots,q_{k+1})),
\]
where
\[
\chi=\bar{\omega}\sum_{j\neq i}\bar{q}_{j}+\bar{q}_{i}\sum_{j>i}\bar{q}_{j},
\]
(see Lemma \ref{Alemma4} in Appendix \ref{Appendix}).

Conversely, let $\{d_{k},\ k\in\mathbb{N}_{0}\}$ be a family of derivations of
${\mathrm{Sym}_{A}(\mathcal{Q},A)}$ such that $d_{k}$ maps ${\mathrm{Sym}%
_{A}^{r}(\mathcal{Q},A)}^{s}$ to ${\mathrm{Sym}_{A}^{r+k}(\mathcal{Q}%
,A)}^{s-k+1}$. For all $a\in A$, and $q_{1},\ldots,q_{k}\in\mathcal{Q}$, put
\[
\{q_{1},\ldots,q_{k-1}|a\}_{k}:=(-)^{a(q_{1}+\cdots+q_{k-1})}(d_{k-1}%
a)(q_{1},\ldots,q_{k-1})\in A
\]
and let $\{q_{1},\ldots,q_{k}\}_{k}\in\mathcal{Q}$ be implicitly defined by
\begin{align*}
\omega(\{q_{1},\ldots,q_{k}\}_{k})  &  :=(-)^{\omega}\sum_{i=1}^{k}%
(-)^{q_{i}(q_{1}+\cdots+q_{i-1})}d_{k-1}(\omega(q_{i}))(q_{1},\ldots
,\widehat{q_{i}},\ldots,q_{k})\\
&  \quad\ -(-)^{\omega}(d_{k-1}\omega)(q_{1},\ldots,q_{k}),
\end{align*}
\ where$\ \omega\in\mathcal{Q}^{\ast}$. Computations in Appendix
\ref{Appendix} (but the other way round) show that i) $\{q_{1},\ldots
,q_{k-1}|a\}$ is symmetric and $A$-linear in the first entries and a graded
derivation in the last one, ii) $\{q_{1},\ldots,q_{k}\}$ is symmetric and
satisfies the Lie-Rinehart property (\ref{LRP}) iii) $J^{k}({}\cdot{}%
,\cdots,{}\cdot{}|{}\cdot{})=0$ and $J^{k}({}\cdot{},\cdots,{}\cdot{})=0$, iff
$E_{k-1}:=\sum_{\ell+m=k-1}[d_{\ell},d_{m}]=0$, $k\in\mathbb{N}$.
\end{proof}

In the smooth setting, the multi-differential algebra $({\mathrm{Sym}%
_{A}(\mathcal{Q},A)},\{d_{k},\ k\in\mathbb{N}_{0}\mathbb{\}})$ determined be
an $LR_{\infty}[1]$-algebra $\mathcal{Q}$ is sometimes called the
\emph{Chevalley-Eilenberg algebra of} $\mathcal{Q}$.

\begin{corollary}
\label{Cor}Any degree $1$, homological derivation $D$ of ${\mathrm{Sym}%
_{A}(\mathcal{Q},A)}$ determines an $LR_{\infty}[1]$-algebra structure on
$(A,\mathcal{Q})$.
\end{corollary}

\begin{proof}
It is enough to define $d_{k}$ as the composition
\[
{\mathrm{Sym}_{A}^{r}(\mathcal{Q},A)}^{s}\overset{D}{\longrightarrow
}{\mathrm{Sym}_{A}(\mathcal{Q},A)}\longrightarrow{\mathrm{Sym}_{A}%
^{r+k}(\mathcal{Q},A)}^{s-k+1}%
\]
where the second arrow is the projection. Then $({\mathrm{Sym}_{A}%
(\mathcal{Q},A)},\{d_{k},\ k\in\mathbb{N}_{0}\})$ is a multi-differential
algebra and $(A,\mathcal{Q})$ gets the structure of an $LR_{\infty}[1]$-algebra.
\end{proof}

\section{Morphisms of Homotopy Lie-Rinehart Algebras\label{MSHLR}}

Let $A$ be an associative, graded commutative, unital algebra. From now on, I
will only consider $LR_{\infty}[1]$-algebras $(A,\mathcal{Q})$, with the extra
finiteness condition that $\mathcal{Q}$ is a projective and finitely generated
$A$-module, without further comments. I will also denote by $D_{\mathcal{Q}%
}=\sum_{k}d_{k}$ the formal derivation of ${\mathrm{Sym}_{A}(\mathcal{Q},A)}$
encoding brackets and anchors in $(A,\mathcal{Q})$. Finally, I occasionally
denote by $p:{\mathrm{Sym}_{A}(\mathcal{Q},A)}\longrightarrow A$ the projection.

The equivalent description of $LR_{\infty}[1]$-algebras in terms of
multi-differential algebras suggests a simple definition of morphisms between
$LR_{\infty}[1]$-algebras $(A,\mathcal{P})$ and $(A,\mathcal{Q})$.

\begin{definition}
\label{Def}A morphism $\phi:(A,\mathcal{P})\longrightarrow(A,\mathcal{Q})$ of
$LR_{\infty}[1]$-algebras is a (degree $0$) morphism of graded, unital
algebras $\psi:{\mathrm{Sym}_{A}(\mathcal{Q},A)}\longrightarrow{\mathrm{Sym}%
_{A}(\mathcal{P},A)}$ such that

\begin{enumerate}
\item $\psi$ is a morphism of multi-differential algebras, i.e., formally,
$\psi\circ D_{\mathcal{Q}}=D_{\mathcal{P}}\circ\psi$ (which is to be
understood component-wise);

\item $p\circ\psi=p$.
\end{enumerate}
\end{definition}

I now re-express it in terms of brackets and anchors. To my knowledge, Formula
(\ref{25}) is presented here for the first time. As a morphism of DG algebras,
$\psi$ is completely determined by its restrictions to $A$ and $\mathcal{Q}%
^{\ast}$. Moreover, composing with the projections ${\mathrm{Sym}%
_{A}(\mathcal{P},A)}\longrightarrow{\mathrm{Sym}_{A}^{k}(\mathcal{P},A)}$, one
get degree $0$ maps
\[
\psi_{k}:A\longrightarrow{\mathrm{Sym}_{A}^{k}(\mathcal{P},A)},\quad\Psi
_{k}:\mathcal{Q}^{\ast}\longrightarrow{\mathrm{Sym}_{A}^{k}(\mathcal{P}%
,A)},\quad k\geq0,
\]
determining $\psi$ in an obvious way. Notice that, by definition, $\psi
_{0}=\mathrm{id}_{A}$ and $\Psi_{0}=0$. The maps $\psi_{k}$ and $\Psi_{k}$ are
not $A$-linear in general. They determine degree $0$ maps
\[
\phi_{k}:{}\mathcal{P}^{\times k}{}\times A\longrightarrow A,\quad\Phi_{k}%
:{}\mathcal{P}^{\times k}{}\longrightarrow\mathcal{Q},\quad k\geq1,
\]
as follows. Let $p_{1},\ldots,p_{k}\in\mathcal{P}$. Put%
\[
\phi_{k}(p_{1},\ldots,p_{k}|a):=(-)^{a(p_{1}+\cdots+p_{k})}\psi_{k}%
(a)(p_{1},\ldots,p_{k}).
\]
Notice that $\phi_{k}$ is $A$-linear and graded symmetric in the $p$'s. On the
other hand, let $\Phi_{k}$ be defined (inductively on $k$) by the implicit
formula:
\begin{align}
\omega(\Phi_{k}(p_{1},\ldots,p_{k}))  &  =\Psi_{k}(\omega)(p_{1},\ldots
,p_{k})\nonumber\\
&  \quad\ -\sum_{\substack{i+j=k\\j>0}}\sum_{\sigma\in S_{i,j}}\alpha
(\sigma,\boldsymbol{p})\psi_{j}(\omega(\Phi_{i}(p_{\sigma(1)},\ldots
,p_{\sigma(i)})))(p_{\sigma(i+1)},\ldots,p_{\sigma(i+j)}), \label{9}%
\end{align}
$\omega\in\mathcal{Q}^{\ast}$.

For instance,
\[
\omega(\Phi_{1}(p))=\Psi_{1}(\omega)(p),
\]%
\begin{align*}
\omega(\Phi_{2}(p_{1},p_{2}))  &  =\Psi_{2}(\omega)(p_{1},p_{2})\\
&  \quad\ -\psi_{1}(\Psi_{1}(\omega)(p_{1}))(p_{2})+{}%
\begin{array}
[c]{c}%
\overset{p_{1},p_{2}}{\leftrightarrow}%
\end{array}
,
\end{align*}
where $\overset{a,b}{\leftrightarrow}$ denotes (Koszul signed) transposition
of $a,b$, and
\begin{align*}
\omega(\Phi_{3}(p_{1},p_{2},p_{3}))  &  =\Psi_{3}(\omega)(p_{1},p_{2},p_{3})\\
&  \quad\ -\psi_{1}(\Psi_{2}(\omega)(p_{1},p_{2}))(p_{3})+{}%
\begin{array}
[c]{c}%
\overset{p_{1},p_{2},p_{3}}{\circlearrowleft}%
\end{array}
\\
&  \quad\ -\psi_{2}(\Psi_{1}(\omega)(p_{1}))(p_{2},p_{3})+{}%
\begin{array}
[c]{c}%
\overset{p_{1},p_{2},p_{3}}{\circlearrowleft}%
\end{array}
\\
&  \quad\ -\psi_{1}(\psi_{1}(\Psi_{1}(\omega)(p_{1}))(p_{2}))(p_{3})+%
\begin{array}
[c]{c}%
\overset{p_{1},p_{2},p_{3}}{\circlearrowleft}%
\end{array}
,
\end{align*}
where $\overset{a,b,c}{\circlearrowleft}$ denotes (Koszul signed) cyclic
permutations of $a,b,c$. Notice that, before one could even write (\ref{9}),
one should prove that the (lower) $\Phi_{i}$'s are well defined, specifically,
that the right hand side $R_{k}(\omega)$ of (\ref{9}) is $A$-linear in
$\omega$. I do this in Appendix \ref{Appendix} (see Lemma \ref{Alemma5} therein).

My next aim is to express condition $\psi\circ D_{\mathcal{Q}}=D_{\mathcal{P}%
}\circ\psi$ in terms of the $\phi$'s and the $\Phi$'s. Notice that for any
morphism of graded algebras $\psi:{\mathrm{Sym}_{A}(\mathcal{Q},A)}%
\longrightarrow{\mathrm{Sym}_{A}(\mathcal{P},A)}$, $[\psi,D]:=\psi\circ
D_{\mathcal{Q}}-D_{\mathcal{P}}\circ\psi$ is a formal derivation along $\psi$.
In particular, $[\psi,D]=0$ iff it vanishes on $A$ and $\mathcal{Q}^{\ast}$.
Let $\omega\in{\mathrm{Sym}_{A}(\mathcal{P},A)}$. In the following I will
denote by $\omega_{k}$ its projection onto ${\mathrm{Sym}_{A}^{k}%
(\mathcal{P},A)}$, $k\geq0$.

\begin{lemma}
\label{Lemma}Let $\omega\in{\mathrm{Sym}_{A}^{r}(\mathcal{Q},A)}$. Then%
\begin{align*}
&  \psi(\omega)_{k}(p_{1},\ldots,p_{k})\\
&  =\sum_{\substack{\ell_{0}+\cdots+\ell_{r}=k\\\ell_{1}\leq\cdots\leq\ell
_{r}}}\sum_{\sigma\in T_{\ell_{0}|\ell_{1},\ldots,\ell_{r}}}(-)^{\chi}%
\alpha(\sigma,\boldsymbol{p})\phi_{\ell_{0}}(p_{\sigma(1)},\ldots
,p_{\sigma(\ell_{0})}|\omega(\Phi_{1,\sigma}(\boldsymbol{p}),\ldots
,\Phi_{r,\sigma}(\boldsymbol{p}))),
\end{align*}
for all $\boldsymbol{p}=(p_{1},\ldots,p_{k})\in\mathcal{P}^{k}$, where
$\chi:=\bar{\omega}(\bar{p}_{\sigma(1)}+\cdots+\bar{p}_{\sigma(\ell_{0})})$,
\[
\Phi_{i,\sigma}(\boldsymbol{p}):=\Phi_{\ell_{i}}(p_{\sigma(\ell_{0}+\ell
_{1}+\cdots+\ell_{i-1}+1)},\ldots,p_{\sigma(\ell_{0}+\ell_{1}+\cdots+\ell
_{i})}),\quad i>0,
\]
and $T_{\ell_{0}|\ell_{1},\ldots,\ell_{r}}$ is the set of permutations of
$\{1,\ldots,k\}$ such that i) $\sigma(\ell_{0}+\ell_{1}+\cdots+\ell
_{i}+1)<\cdots<\sigma(\ell_{0}+\ell_{1}+\cdots+\ell_{i}+\ell_{i+1})$, and ii)
$\sigma(\ell_{0}+\ell_{1}+\cdots+\ell_{i}+1)<\sigma(\ell_{0}+\ell_{1}%
+\cdots+\ell_{i}+\ell_{i+1}+1)$ whenever $\ell_{i}=\ell_{i+1}$, $i>0$.
\end{lemma}

\begin{proof}
see Appendix \ref{Appendix}.
\end{proof}

Now I want to characterize morphisms of $LR_{\infty}[1]$-algebras. If $a\in
A$, then $\psi(D_{\mathcal{Q}}a)=D_{\mathcal{P}}\psi(a)$ means that
\begin{equation}
\psi(D_{\mathcal{Q}}a)_{k}(p_{1},\ldots,p_{k})=(D_{\mathcal{P}}\psi
(a))_{k}(p_{1},\ldots,p_{k}),\quad p_{1},\ldots,p_{k}\in\mathcal{P},
\label{28}%
\end{equation}
for all $k$. Since both hand sides of (\ref{28}) are graded $K$-multilinear,
and graded symmetric, \emph{they coincide iff they coincide on equal, even
arguments} $p_{1}=\cdots=p_{k}=p$ (in the following, and especially in
Appendix \ref{Appendix}, to prove that two multilinear, graded symmetric maps
are equal, I will often use the trick of evaluating them on equal, even
arguments, without further comments). Now, $\psi(D_{\mathcal{Q}}a)_{k}%
=\sum_{m}\psi(d_{m}a)_{k}$. In its turn, in view of Lemma \ref{Lemma},
$\psi(d_{m}a)_{k}$ is given by%
\[
\psi(d_{m}a)_{k}(p^{k})=\sum_{\substack{\ell_{0}+\cdots+\ell_{m}=k\\\ell
_{1}\leq\cdots\leq\ell_{m}}}C(\ell_{0}|\ell_{1},\ldots,\ell_{m})\phi_{\ell
_{0}}(p^{\ell_{0}}|\{\Phi_{\ell_{1}}(p^{\ell_{1}}),\ldots,\Phi_{\ell_{m}%
}(p^{\ell_{m}})|a\}),
\]
where $C(\ell_{0}|\ell_{1},\ldots,\ell_{m})$ is the cardinality of
$T_{\ell_{0}|\ell_{1},\ldots,\ell_{m}}$. On the other hand, $(D_{\mathcal{P}%
}\psi(a))_{k}=\sum_{m}(d_{m}\psi(a))_{k}$ and a short computation shows that
\[
(d_{m}\psi(a))_{k}(p^{k})=\tbinom{k}{m}\{p^{m}|\phi_{k-m}(p^{k-m}%
|a)\}-\tbinom{k}{m+1}\phi_{k-m}(\{p^{m+1}\},p^{k-m-1}|a).
\]

I conclude that
\begin{align}
&  \sum_{m}\sum_{\substack{\ell_{0}+\cdots+\ell_{m}=k\\\ell_{1}\leq\cdots
\leq\ell_{m}}}C(\ell_{0}|\ell_{1},\ldots,\ell_{m})\phi_{\ell_{0}}(p^{\ell_{0}%
}|\{\Phi_{\ell_{1}}(p^{\ell_{1}}),\ldots,\Phi_{\ell_{m}}(p^{\ell_{m}%
})|a\})\nonumber\\
&  =\sum_{m}\tbinom{k}{m}\{p^{m}|\phi_{k-m}(p^{k-m}|a)\}-\sum_{m}\tbinom
{k}{m+1}\phi_{k-m}(\{p^{m+1}\},p^{k-m-1}|a), \label{11}%
\end{align}
for all $k$.

Similarly, let $\omega\in\mathcal{Q}^{\ast}$. Then $\psi(D_{\mathcal{Q}}%
\omega)=D_{\mathcal{P}}\psi(\omega)$ means that $\psi(D_{\mathcal{Q}}%
\omega)_{k}=(D_{\mathcal{P}}\psi(\omega))_{k}$ for all $k$, which can be
compactly rewritten in the form
\begin{equation}
\sum_{m+\ell=k}\tbinom{m+\ell}{m}\phi_{\ell}(p^{\ell}|\omega(K_{\Phi}%
^{m}(p^{m})))=0, \label{14}%
\end{equation}
where
\[
K_{\Phi}^{n}(p^{n})=\sum_{m}\sum_{\substack{\ell_{1}+\cdots+\ell_{m}%
=n\\\ell_{1}\leq\cdots\leq\ell_{m}}}C(\ell_{1},\ldots,\ell_{m})\{\Phi
_{\ell_{1}}(p^{\ell_{1}}),\ldots,\Phi_{\ell_{m}}(p^{\ell_{m}})\}-\sum
_{m+r=n}\tbinom{m+r}{m}\Phi_{r+1}(\{p^{m}\},p^{r}),
\]
and $C(\ell_{1},\ldots,\ell_{m})$ is the cardinality of $S_{\ell_{1}%
,\ldots,\ell_{m}}^{<}$ (see Lemma \ref{Alemma6} in Appendix \ref{Appendix}).
From Formula (\ref{14}) it is easy to see, by induction on $r$, that $K_{\Phi
}^{r}$ vanishes for all $r$. I have thus proved the following

\begin{theorem}
\label{TheorMor}A (degree $0$) morphism of graded, unital algebras
$\psi:{\mathrm{Sym}_{A}(\mathcal{Q},A)}\longrightarrow{\mathrm{Sym}%
_{A}(\mathcal{P},A)}$ such that $p\circ\psi=p$ is a morphism of $LR_{\infty
}[1]$-algebras $\phi:(A,\mathcal{P})\longrightarrow(A,\mathcal{Q})$ iff

\begin{enumerate}
\item $\phi=\{\phi_{k},\ k\in\mathbb{N}\}$ is a morphism of $L_{\infty}[1]$-algebras

\item Formula
\begin{align}
&  \sum_{\substack{m\\\ell_{0}+\cdots+\ell_{m}=k\\\ell_{1}\leq\cdots\leq
\ell_{m-1}}}\sum_{\sigma\in T_{\ell_{0}|\ell_{1},\ldots,\ell_{m}}%
}(-)^{p_{\sigma(1)}+\cdots+p_{\sigma(\ell_{0})}}\alpha(\sigma,\boldsymbol{p}%
)\phi_{\ell_{0}}(p_{\sigma(1)},\ldots,p_{\sigma(\ell_{0})}|\{\Phi_{1,\sigma
}(\boldsymbol{p}),\ldots,\Phi_{m,\sigma}(\boldsymbol{p})|a\})\nonumber\\
&  =\sum_{\ell+m=k}\sum_{\sigma\in S_{\ell,m}}\alpha(\sigma,\boldsymbol{p}%
)\{p_{\sigma(1)},\ldots,p_{\sigma(\ell)}|\phi_{m}(p_{\sigma(\ell+1)}%
,\ldots,p_{\sigma(\ell+m)}|a)\}\nonumber\\
&  \quad\ -\sum_{\ell+m=k}\sum_{\sigma\in S_{\ell,m}}\alpha(\sigma
,\boldsymbol{p})\phi_{m+1}(\{p_{\sigma(1)},\ldots,p_{\sigma(\ell)}%
\},p_{\sigma(\ell+1)},\ldots,p_{\sigma(\ell+m)}|a) \label{25}%
\end{align}
holds for all $\boldsymbol{p}=(p_{1},\ldots,p_{k})\in\mathcal{P}^{\times k}$,
$a\in A$, and $k\in\mathbb{N}_{0}$.
\end{enumerate}
\end{theorem}

Formula (\ref{25}) could be hardly guessed without the description of
$LR_{\infty}[1]$-algebras in terms of multi-differential algebras.

\part{Geometric Applications}

\label{P2}

\section{Form-Valued Vector Fields\label{FVVF}}

Le $M$ be a smooth manifold. A \emph{form-valued vector field} on $M$ is an
element of the (graded) $\Lambda(M)$-module obtained from $\mathfrak{X}(M)$ by
extension of scalars, i.e., $\Lambda(M)\otimes\mathfrak{X}(M)$. Notice that a
form-valued vector field $Z$ on $M$ may be understood as a derivation of the
algebra $C^{\infty}(M)$ with values in the $C^{\infty}(M)$-module $\Lambda
(M)$. It may be understood also as a $\mathfrak{X}(M)$-valued, skew-symmetric,
multilinear map on $\mathfrak{X}(M)$. In the following I will take both points
of view. Form-valued vector fields inherit a rich calculus which I call
\emph{Fr\"{o}licher-Nijenhuis calculus}. In this section, I briefly review it,
referring to \cite{fn56} and \cite{m08} for details.

\begin{theorem}
[see, for instance, \cite{m08}]Let $Z\in\Lambda(M)\otimes\mathfrak{X}(M)$.
There exist unique graded derivations $i_{Z},L_{Z}:\Lambda(M)\longrightarrow
\Lambda(M)$ (called the \emph{insertion of} $Z$, and the \emph{Lie derivative
along} $Z$, respectively) such that

\begin{enumerate}
\item $i_{Z}$ is $C^{\infty}(M)$-linear and $i_{Z}df=Z(f)$ for all $f\in
C^{\infty}(M)$,

\item $L_{Z}$ commutes (in the graded sense) with $d$ and $L_{Z}f=Z(f)$ for
all $f\in C^{\infty}(M)$.
\end{enumerate}

Conversely, let $\Delta:\Lambda(M)\longrightarrow\Lambda(M)$ be a graded
derivation. There exist unique form-valued vector fields $Z,Y$ such that
\[
\Delta=i_{Z}+L_{Y}.
\]

\end{theorem}

Notice that, if $Z\in\Lambda^{k}(M)\otimes\mathfrak{X}(M)$ then $i_{Z}$
(resp., $L_{Z}$) is homogeneous of degree $k-1$ (resp., $k$). Conversely if
$\Delta=i_{Z}+L_{Y}:\Lambda(M)\longrightarrow\Lambda(M)$ is a homogeneous
derivation of degree $\ell$, then $Z\in\Lambda^{\ell+1}(M)\otimes
\mathfrak{X}(M)$ and $Y\in\Lambda^{\ell}(M)\otimes\mathfrak{X}(M)$. For
$Z\in\Lambda(M)\otimes\mathfrak{X}(M)$, abusing the notation, I also denote by
$i_{Z}$ the $C^{\infty}(M)$-linear map
\[
i_{Z}\otimes\operatorname{id}:\Lambda(M)\otimes\mathfrak{X}(M)\longrightarrow
\Lambda(M)\otimes\mathfrak{X}(M).
\]

Derivations of $\Lambda(M)$ of the form $i_{Z}$ (resp., $L_{Z}$) form a Lie
subalgebra. Namely, let $Z_{1},Z_{2}\in\Lambda(M)\otimes\mathfrak{X}(M)$. Then
there exists a unique $Z\in\Lambda(M)\otimes\mathfrak{X}(M)$ such that
\[
\lbrack i_{Z_{1}},i_{Z_{2}}]=i_{Z}\quad\text{(resp., }[L_{Z_{1}},L_{Z_{2}%
}]=L_{Z}\text{).}%
\]
$Z$ is denoted by $[Z_{1},Z_{2}]_{\mathrm{nr}}$ (resp., $[\![Z_{1},Z_{2}]\!]$)
and called the \emph{Nijenhuis-Richardson} (resp., \emph{Fr\"{o}licher-Nijenhuis})
\emph{bracket} of $Z_{1}$ and $Z_{2}$. The brackets $[{}\cdot{},{}\cdot
{}]_{\mathrm{nr}}$ (resp., $[\![{}\cdot{},{}\cdot{}]\!]$) give $(\Lambda
(M)\otimes\mathfrak{X}(M))[1]$ (resp., $\Lambda(M)\otimes\mathfrak{X}(M)$) the
structure of a graded Lie algebra. In particular, $[{}\cdot{},{}\cdot
{}]_{\mathrm{nr}}$ (resp., $[\![{}\cdot{},{}\cdot{}]\!]$) satisfies a suitable
graded Jacobi identity.

\begin{theorem}
Let $\omega\in\Lambda(M)$, and $X,Y,Z\in\Lambda(M)\otimes\mathfrak{X}(M)$ be
homogeneous elements. The following formulas hold
\begin{align}
i_{\omega Z}  &  =\omega\,i_{Z},\nonumber\\
L_{\omega Z}  &  =\omega L_{Z}+(-)^{\omega+Z}d\omega\,i_{Z},\label{f4}\\
\lbrack i_{Z},d]  &  =L_{Z},\nonumber\\
\lbrack i_{Z},L_{Y}]  &  =L_{i_{Z}Y}-(-)^{Y}i_{[\![Z,Y]\!]},\label{f5}\\
\lbrack Z,Y]_{\mathrm{nr}}  &  =i_{Z}Y-(-)^{(Z-1)(Y-1)}i_{Y}Z,\nonumber\\
\lbrack\omega Z,Y]_{\mathrm{nr}}  &  =\omega\lbrack Z,Y]_{\mathrm{nr}%
}-(-)^{(\omega+Z-1)(Y-1)}(i_{Y}\omega)Z,\label{f6}\\
\lbrack\![\omega Z,Y]\!]  &  =\omega\lbrack\![Z,Y]\!]-(-)^{(\omega+Z)Y}%
(L_{Y}\omega)Z+(-)^{\omega+Z}d\omega\,i_{Z}Y,\label{f2}\\
i_{X}[\![Z,Y]\!]  &  =[\![i_{X}Z,Y]\!]+(-)^{(X-1)Z}[\![Z,i_{X}Y]\!]+(-)^{Z}%
i_{[\![X,Z]\!]}Y-(-)^{Y(Z-1)}i_{[\![X,Y]\!]}Z, \label{f3}%
\end{align}
and
\begin{align}
&  [\![X,[Z,Y]_{\mathrm{nr}}]\!]\nonumber\\
&  =[[\![X,Z]\!],Y]_{\mathrm{nr}}+(-)^{X(Z-1)}\left(
[Z,[\![X,Y]\!]]_{\mathrm{nr}}-[\![i_{Z}X,Y]\!]\right)  +(-)^{(X+Z-1)(Y-1)}%
[\![i_{Y}X,Z]\!]. \label{f1}%
\end{align}

\end{theorem}

Below, I will often use formulas in the above theorem, sometimes without any comment.

\begin{remark}
\label{RemAlg}If $(C^\infty (M),L)$ is the Lie-Rinehart algebra of sections of a Lie
algebroid over $M$, then there is an analogue of the Fr\"{o}licher-Nijenhuis
calculus on $\Lambda^{\bullet}L^{\ast}\otimes L$. In particular, elements
$Z\in\Lambda^{\bullet}L^{\ast}\otimes L$ determine derivations $i_{Z}$,
$L_{Z}$ of $\Lambda^{\bullet}L^{\ast}$, satisfying analogues of formulas
(\ref{f4})--(\ref{f1}) (see, for instance, \cite{g13}).
\end{remark}

\section{Homological Algebra of Foliations\label{HAF}}

Let $M$ be a smooth manifold and $C$ an involutive $n$-dimensional
distribution on it. Let $%
\mathscr{F}%
$ be the integral foliation of $C$. I will denote by $C\mathfrak{X}$ the
submodule of $\mathfrak{X}(M)$ made of vector fields in $C$. Moreover,
following A.{} Vinogradov (and his school) \cite{v84,v98,v01,b...99}, I will
denote by $C\Lambda^{1}:=C\mathfrak{X}^{\bot}\subset\Lambda^{1}(M)$ the
annihilator of $C\mathfrak{X}$. Put
\begin{align*}
\overline{\mathfrak{X}}  &  :=\mathfrak{X}(M)/C\mathfrak{X},\\
\overline{\Lambda}{}^{1}  &  :=\Lambda^{1}(M)/C\Lambda^{1}.
\end{align*}
Then $C\Lambda^{1}\simeq\overline{\mathfrak{X}}{}^{\ast}$ and $\overline
{\Lambda}{}^{1}\simeq C\mathfrak{X}^{\ast}$. In view of the Fr\"{o}benius
theorem, there always exist coordinates $\ldots,x^{i},\ldots,u^{\alpha}%
,\ldots$, $i=1,\ldots,n$, $\alpha=1,\ldots,\dim M-n$, adapted to $C$, i.e.,
such that $C\mathfrak{X}$ is locally spanned by $\ldots,\partial_{i}%
:=\partial/\partial x^{i},\ldots$ and $C\Lambda^{1}$ is locally spanned by
$\ldots,du^{\alpha},\ldots$. Consider the Chevalley-Eilenberg DG algebra
$(\overline{\Lambda},\overline{d})$ associated to the Lie algebroid
$C\mathfrak{X}$, i.e., $\overline{\Lambda}$ is the exterior algebra of
$\overline{\Lambda}{}^{1}$ and
\begin{align*}
&  (\overline{d}\lambda)(X_{1},\dots,X_{k+1})\\
&  =\sum_{i}(-)^{i+1}X_{i}(\lambda(\ldots,\widehat{X}_{i},\ldots))+\sum
_{i<j}(-)^{i+j}\lambda([X_{i},X_{j}],\ldots,\widehat{X}_{i},\ldots,\widehat
{X}_{j},\ldots)
\end{align*}
where $\lambda\in\overline{\Lambda}{}^{k}$ is understood as a $C^{\infty}%
(M)$-valued, $k$-multilinear, skew-symmetric map on $C\mathfrak{X}$ and
$X_{1},\ldots,X_{k+1}\in C\mathfrak{X}$. The DG algebra $(\overline{\Lambda
},\overline{d})$ is the quotient of $(\Lambda(M),d)$ over the differentially
closed ideal generated by $C\Lambda^{1}$. In particular, it
is generated by degree $0$, and $\overline{d}$-exact degree $1$ elements.

In the following, I write $\omega\longmapsto\overline{\omega}$ the projection
$\Lambda(M)\longrightarrow\overline{\Lambda}$.

The Lie algebroid $C\mathfrak{X}$ acts on $\overline{\mathfrak{X}}$ via the
\emph{Bott connection}. Namely, write $X\longmapsto\overline{X}$ the
projection $\mathfrak{X}(M)\longrightarrow\overline{\mathfrak{X}}$. Then
\[
X{}\cdot{}\overline{Y}:=\overline{[X,Y]}\in\overline{\mathfrak{X}},\quad X\in
C\mathfrak{X},\text{ }Y\in\mathfrak{X}(M).
\]
Accordingly, there is a differential $(\overline{\Lambda},\overline{d}%
)$-module $(\overline{\Lambda}\otimes\overline{\mathfrak{X}},\overline{d})$
whose differential is given by the usual Chevalley-Eilenberg formula:
\begin{align*}
&  (\overline{d}Z)(X_{1},\dots,X_{k+1})\\
&  =\sum_{i}(-)^{i+1}X_{i}\cdot Z(\ldots,\widehat{X}_{i},\ldots)+\sum
_{i<j}(-)^{i+j}Z([X_{i},X_{j}],\ldots,\widehat{X}_{i},\ldots,\widehat{X}%
_{j},\ldots),
\end{align*}
where $Z\in\overline{\Lambda}{}^{k}\otimes\overline{\mathfrak{X}}$ is
understood as a $\overline{\mathfrak{X}}$-valued, $k$-multilinear,
skew-symmetric map on $C\mathfrak{X}$, and $X_{1},\ldots,X_{k+1}\in
C\mathfrak{X}$. The tensor product
\[
\Lambda(M)\otimes\mathfrak{X}(M)\longrightarrow\overline{\Lambda}%
\otimes\overline{\mathfrak{X}}.
\]
of projections $\Lambda(M)\longrightarrow\overline{\Lambda}$ and
$\mathfrak{X}(M)\longrightarrow\overline{\mathfrak{X}}$ will be written
$Z\longmapsto\overline{Z}$.

\begin{remark}
\label{Rema} \label{RemExt}The $\overline{d}$ differentials in $\overline
{\Lambda}$ and $\overline{\mathfrak{X}}$ can be uniquely extended to the whole
tensor algebra%
\[
\bigoplus_{i,j}\overline{\Lambda}\otimes\overline{\mathfrak{X}}{}^{\otimes
i}\otimes(C\Lambda^{1})^{\otimes j},
\]
by requiring Leibniz rules with respect to tensor products and contractions.
\end{remark}

The zeroth cohomology $H^{0}(\overline{\Lambda},\overline{d})$ is made of
functions on $M$ which are constant along the leaves of $C$ and, therefore,
elements in it are naturally interpreted as functions on the \textquotedblleft
space of leaves\textquotedblright. In secondary calculus, one is also
concerned with lower dimensional integral submanifolds of $C$. In this
respect, it is natural to understand the whole
\[
\boldsymbol{C}_{%
\mathscr{F}%
}^{\infty}:=H(\overline{\Lambda},\overline{d})
\]
as algebra of functions over the \textquotedblleft space of integral
manifolds\textquotedblright. Notice that $\boldsymbol{C}_{%
\mathscr{F}%
}^{\infty}$ is nothing but the \emph{leaf-wise (de Rham) cohomology} of $%
\mathscr{F}%
$. The zeroth cohomology $H^{0}(\overline{\Lambda}\otimes\overline
{\mathfrak{X}},\overline{d})$ is made of vector fields on $M$ preserving $C$
(modulo vector fields in $C$) and, therefore, elements in it are naturally
interpreted as vector fields on the \textquotedblleft space of
leaves\textquotedblright. Just as above, it is natural to understand the
whole
\[
\boldsymbol{\mathfrak{X}}_{%
\mathscr{F}%
}:=H(\overline{\Lambda}\otimes\overline{\mathfrak{X}},\overline{d})
\]
as Lie algebra of vector fields over the \textquotedblleft space of integral
manifolds\textquotedblright. The geometric interpretation of cohomologies of
$\overline{d}$ is very fruitful and far reaching \cite{v01,v09}. The following
theorem supports this interpretation.

Denote by $[\theta]$ the cohomology class in either $\boldsymbol{C}_{%
\mathscr{F}%
}^{\infty}$ or $\boldsymbol{\mathfrak{X}}_{%
\mathscr{F}%
}$ of a cocycle $\theta$ in either $\overline{\Lambda}$ or $\overline{\Lambda
}\otimes\overline{\mathfrak{X}}$, respectively.

\begin{theorem}
\label{Theorem1} The pair $(\boldsymbol{C}_{%
\mathscr{F}%
}^{\infty},\boldsymbol{\mathfrak{X}}_{%
\mathscr{F}%
})$ possesses a canonical structure of graded Lie-Rinehart algebra. Namely,

\begin{enumerate}
\item \label{1.1}$\boldsymbol{\mathfrak{X}}_{%
\mathscr{F}%
}$ has a graded Lie algebra structure $\boldsymbol{[}{}\cdot{},{}\cdot
{}\boldsymbol{]}$ given by
\[
\boldsymbol{[}[\overline{Z}_{1}],[\overline{Z}_{2}]\boldsymbol{]}%
:=[\overline{[\![Z_{1},Z_{2}]\!]}]\in\boldsymbol{\mathfrak{X}}_{%
\mathscr{F}%
},\quad\lbrack\overline{Z}_{1}],[\overline{Z}_{2}]\in\boldsymbol{\mathfrak{X}%
}_{%
\mathscr{F}%
},\quad Z_{1},Z_{2}\in\Lambda(M)\otimes\mathfrak{X}(M),
\]

\item \label{2.1}$\boldsymbol{C}_{%
\mathscr{F}%
}^{\infty}$ has a graded Lie module structure over $(\boldsymbol{\mathfrak{X}%
}_{{%
\mathscr{F}%
}},\boldsymbol{[}{}\cdot{},{}\cdot{}\boldsymbol{]})$ given by
\[
\lbrack\overline{Z}]{}\cdot{}[\overline{\lambda}]:=[\overline{L_{Z}\lambda
}]\in\boldsymbol{C}_{%
\mathscr{F}%
}^{\infty},\quad\lbrack\overline{\lambda}]\in\boldsymbol{C}^{\infty},\text{
}[\overline{Z}]\in\boldsymbol{\mathfrak{X}}_{%
\mathscr{F}%
},\quad\lambda\in\Lambda(M),\quad Z\in\Lambda(M)\otimes\mathfrak{X}(M),
\]

\item \label{3.1}$\boldsymbol{\mathfrak{X}}_{%
\mathscr{F}%
}$ has a graded module structure over $\boldsymbol{C}_{%
\mathscr{F}%
}^{\infty}$ given by
\[
\lbrack\overline{\lambda}]{}\cdot{}[\overline{Z}]:=[\overline{\lambda Z}%
]\in\boldsymbol{\mathfrak{X}},\quad\lbrack\overline{\lambda}]\in
\boldsymbol{C}_{%
\mathscr{F}%
}^{\infty},\text{ }[\overline{Z}]\in\boldsymbol{\mathfrak{X}}_{%
\mathscr{F}%
},\quad\lambda\in\Lambda(M),\quad Z\in\Lambda(M)\otimes\mathfrak{X}(M),
\]

\item \label{4.1}$\boldsymbol{\mathfrak{X}}_{%
\mathscr{F}%
}$ acts on $\boldsymbol{C}_{%
\mathscr{F}%
}^{\infty}$ via graded derivations, and the formula
\[
(\boldsymbol{f}\cdot\boldsymbol{X}{}){}\cdot{}\boldsymbol{g}:=\boldsymbol{f}%
\cdot(\boldsymbol{X}{}{}\cdot{}\boldsymbol{g})\boldsymbol{,}%
\]
holds for all $\boldsymbol{f},\boldsymbol{g}\in\boldsymbol{C}_{%
\mathscr{F}%
}^{\infty}$, and $\boldsymbol{X}\in\boldsymbol{\mathfrak{X}}_{%
\mathscr{F}%
}$,

\item \label{5.1}the formula
\[
\boldsymbol{[X},\boldsymbol{fY]}=(\boldsymbol{X}{}\cdot{}\boldsymbol{f}%
){}\cdot{}\boldsymbol{Y}+(-)^{\boldsymbol{Xf}}\boldsymbol{f}\cdot
\boldsymbol{{}[X},\boldsymbol{Y]},
\]
holds for all $\boldsymbol{X},\boldsymbol{Y}\in\boldsymbol{\mathfrak{X}}_{%
\mathscr{F}%
}$, $\boldsymbol{f}\in\boldsymbol{C}_{%
\mathscr{F}%
}^{\infty}$.
\end{enumerate}
\end{theorem}

In Section \ref{SHLRF}, I show that the graded Lie-Rinehart algebra of Theorem
\ref{Theorem1} actually comes from an $LR_{\infty}[1]$\emph{-algebra}
structure on $(\overline{\Lambda},\overline{\Lambda}\otimes\overline
{\mathfrak{X}}[1])$ (see also \cite{h05,j12}), according to the following

\begin{theorem}
The pair $(\overline{\Lambda},\overline{\Lambda}\otimes\overline{\mathfrak{X}%
}[1])$ possesses a structure of $LR_{\infty}[1]$-algebra. Namely,

\begin{enumerate}
\item \label{1.}$\overline{\Lambda}\otimes\overline{\mathfrak{X}}[1]$ has an
$L_{\infty}[1]$-algebra structure, denoted $\mathscr{L}=\{\{{}\cdot{}%
,\cdots,{}\cdot{}\}_{k},\ k\in\mathbb{N}\}$, such that $\{{}\cdot{}%
\}_{1}=\overline{d}$,

\item \label{2.}$\overline{\Lambda}$ has an $L_{\infty}[1]$-module structure
over $(\overline{\Lambda}\otimes\overline{\mathfrak{X}}[1],\mathscr{L})$,
denoted $\{\{{}\cdot{}{},\cdots,{}\cdot{}|{}\cdot{}\}_{k},\ k\in\mathbb{N}\}$,
such that $\{|{}\cdot{}{}\}_{1}=\overline{d}$,

\item \label{3.}$\overline{\Lambda}\otimes\overline{\mathfrak{X}}[1]$ has a
graded $\overline{\Lambda}$-module structure,

\item \label{4.}$\{{}\cdot{}{},\cdots,{}\cdot{}|{}\cdot{}\}_{k}:(\overline
{\Lambda}\otimes\overline{\mathfrak{X}}[1])^{k-1}\times\overline{\Lambda
}\longrightarrow\overline{\Lambda}$ is a derivation in the last argument and
it is $\overline{\Lambda}$-multilinear in the first $k$ entries,
$k\in\mathbb{N}$,

\item \label{5.}the formula
\[
\{Z_{1},\ldots,Z_{k-1},\lambda Z_{k}\}=\{Z_{1},\ldots,Z_{k-1}|\lambda\}{}%
\,{}Z_{k}+(-)^{\lambda\boldsymbol{(}Z_{1}+\cdots+Z_{k-1}+1)}\lambda
\{Z_{1},\ldots,Z_{k-1},Z_{k}\}
\]
holds for all $Z_{1},\ldots,Z_{k-1},Z_{k}\in\overline{\Lambda}\otimes
\overline{\mathfrak{X}}$, $\lambda\in\overline{\Lambda}$, and $k\in\mathbb{N}$.
\end{enumerate}

Moreover, structures in \ref{1.}, \ref{2.}, \ref{3.}, induce structures in
\ref{1.1}, \ref{2.1}, \ref{3.1} of Theorem \ref{Theorem1} in cohomology, up to
a sign (due to the chosen sign conventions), and properties \ref{4.}, \ref{5.}
imply properties \ref{4.1}, \ref{5.1} of Theorem \ref{Theorem1}, respectively.
\end{theorem}


When I was finalizing the first version of this paper, Chen, Sti\'{e}non, and
Xu published an e-print \cite{csx12} were they present similar results in a
much wider context. Namely, they consider what they call a \emph{Lie pair},
i.e., a Lie algebroid $L$ (generalizing $TM$ in this paper) with a Lie
subalgebroid $L_{1}$ (generalizing $C$ in this paper). There is a differential
in $\Lambda^{\bullet}L_{1}^{\ast}\otimes L/L_{1}$ (generalizing the
differential $\overline{d}$ in $\overline{\Lambda}\otimes\overline
{\mathfrak{X}}$) and a Lie algebra structure on cohomology. The authors of
\cite{csx12} prove that such a Lie algebra comes from a SH Leibniz algebra
structure on $\Lambda^{\bullet}L_{1}^{\ast}\otimes L/L_{1}$ defined by means
of: 1) a splitting of the inclusion of modules $L_{1}\subset L$ and 2) an
$L$-connection extending the canonical $L_{1}$-connection in $L/L_{1}$.
Finally, such an SH\ Leibniz algebra is a genuine $L_{\infty}[1]$-algebra if
the inclusion $L_{1}\subset L$ is split via another Lie subalgebroid. They
prove similar results for a general $L_{1}$-module. Obviously, their framework
encompasses mine. However, their results do not encompass mine for many
reasons: i) the SH structure $\mathscr{Q}$ on $\overline{\Lambda}%
\otimes\overline{\mathfrak{X}}$ described in this paper is always a true
$L_{\infty}[1]$-algebra, and not just a SH Leibniz algebra; ii) I define
$\mathscr{Q}$ by just a splitting of the inclusion $C\mathfrak{X}%
\subset\mathfrak{X}(M)$, and then prove that it is independent of the
splitting up to isomorphisms; iii) the SH structure $\mathscr{Q}$
possesses only one higher homotopy (a third level one); iv) when the splitting
is made via another involutive distribution, $(\overline{\Lambda}%
\otimes\overline{\mathfrak{X}},\mathscr{Q})$ actually becomes a DG Lie algebra
(the higher homotopy vanishes) up to a sign (due to the chosen sign
conventions); v) I also discuss the question: where does the structure of a
Lie-Rinehart algebra on $(\boldsymbol{C}_{%
\mathscr{F}%
}^{\infty},\boldsymbol{\mathfrak{X}}_{%
\mathscr{F}%
})$ comes from? And not only the questions: where does the Lie algebra (resp.,
Lie module) structure on $\boldsymbol{\mathfrak{X}}_{%
\mathscr{F}%
}$ (resp., $\boldsymbol{C}
_{%
\mathscr{F}%
}^{\infty}$) comes from?

I have to mention also that, while I was preparing a revised version of this
paper (already published as e-print arXiv:1204.2467v1), Ji published an
e-print \cite{j12} were he basically presents, among other things, (part of)
my results, in the already mentioned wider context of \cite{csx12}. Ji's main
aim is to discuss deformations of Lie pairs (in the sense of \cite{csx12}).
Notice that, i) he defines only the $L_{\infty}[1]$-algebra structure
determined by a Lie pair, and not the $LR_{\infty}[1]$-algebra structure; ii)
he does this via T.{} Voronov's derived bracket formalism \cite{v05}, however
(it is not hard to see that) his $L_{\infty}[1]$-algebra coincides with the
one of Section \ref{SHLRF} (see the first appendix); ii) he does not discuss
the dependence on the choice of the above mentioned splitting. Notice also
that the methods of this paper, including the use of the
Fr\"{o}licher-Nijenhuis calculus (see Remark \ref{RemAlg}), can be
straightforwardly generalized to the case of a Lie pair more general than an
involutive distribution. However, I prefer to stay on the latter case. Indeed,
it is the relevant one for applications in secondary calculus which is the
ultimate goal of the paper.

Finally, I stress again that the existence of a quasi Lie-Rinehart algebra
associated to a Lie pair had been already proved by Huebschmann in 2005
\cite{h05}. The quasi Lie-Rinehart algebra of Huebschmann coincides with the
$LR_{\infty}[1]$-algebra in Section \ref{SHLRF}. However, I decided to present
again its derivation in this paper for various reasons. From an algebraic
point of view, for the reasons already discussed in Remarks \ref{Hue1},
\ref{Hue2}. From a geometric point of view, because 1) the presentation in
terms of Fr\"{o}licher-Nijenhuis calculus is somewhat more explicit and easy
to work with, 2) I complement it with a proof of canonicity (see Section
\ref{CS}), 3) I relate it to the work of Oh and Park \cite{op05} (see Section
\ref{SHLRPSF}).

\section{{Adding a Complementary Distribution to an Involutive Distribution}%
\label{GSF}}

The exact sequence
\[
0\longrightarrow C\mathfrak{X}\longrightarrow\mathfrak{X}(M)\longrightarrow
\overline{\mathfrak{X}}\longrightarrow0
\]
splits. The datum of a splitting is equivalent to the datum of a distribution
$V$ complementary to $C$. From now on, fix such a distribution. I will always
identify $\overline{\mathfrak{X}}$ (resp., $\overline{\Lambda}{}^{1}$) with
the corresponding submodule in $\mathfrak{X}(M)$ (resp., $\Lambda^{1}(M)$)
determined by $V$. The triple $(C^{\infty}(M), C\mathfrak{X}, \overline
{\mathfrak{X}})$ is actually a special instance of Huebschmann's
\emph{Lie-Rinehart triple} \cite{h05}.

Let $C\Lambda:=\bigoplus_{k}C\Lambda^{k}\subset\Lambda(M)$ be the $C^{\infty
}(M)$-subalgebra generated by $C\Lambda^{1}$, $C\Lambda^{k}:=\Lambda
^{k}C\Lambda^{1}$. Then
\begin{equation}
\Lambda(M)\simeq\Lambda^{\bullet}\Lambda^{1}(M)\simeq\Lambda^{\bullet
}(\overline{\Lambda}{}^{1}\oplus C\Lambda^{1})\simeq\overline{\Lambda}\otimes
C\Lambda. \label{Fact}%
\end{equation}
In the following, I will identify $\Lambda(M)$ and $\overline{\Lambda}\otimes
C\Lambda$. The transversal distribution $V\simeq TM/C$ is locally spanned by
vector fields $\ldots,V_{\alpha},\ldots$ of the form $V_{\alpha}%
:=\partial/\partial u^{\alpha}+V_{\alpha}^{i}\partial_{i}$, $\alpha
=1,\ldots\dim M-n$, for some local functions $\ldots,V_{\alpha}^{i},\ldots$,
and its annihilator $V^{\bot}\simeq C^{\ast}$ is locally spanned by
differential forms $\ldots,d^{C}\!x^{i}:=dx^{i}-V_{\alpha}^{i}du^{\alpha
},\ldots$.

Denote by $P^{C},P^{V}\in\Lambda^{1}(M)\otimes\mathfrak{X}(M)$ the projectors
onto $C,V$, and by $d^{C},d^{V}$ Lie derivatives of differential forms along
them, respectively (this is consistent with our previous notations). The form
valued vector field $P^{C}\ $belongs to $\overline{\Lambda}{}^{1}\otimes
C\mathfrak{X}$, and it is locally given by
\[
P^{C}=d^{C}\!x^{i}\otimes\partial_{i}.
\]
Similarly, $P^{V}\in C\Lambda^{1}\otimes\overline{\mathfrak{X}}$, and it is
locally given by
\[
P^{V}=du^{\alpha}\otimes V_{\alpha}.
\]
By definition, $P^{C}+P^{V}=\mathbb{I}$, the identity in $\Lambda
^{1}(M)\otimes\mathfrak{X}(M)$, so that $d^{C}+d^{V}=d$. The \emph{curvature}
of the splitting $V$ is, by definition,
\[
R:=\dfrac{1}{2}[\![P^{C},P^{C}]\!]\in\Lambda^{2}(M)\otimes\mathfrak{X}(M).
\]
It is easy to see that $R\in C\Lambda^{2}\otimes C\mathfrak{X}$. Moreover,
$P^{C}$, $P^{V}$ and $R$ generate a Lie subalgebra of $(\Lambda(M)\otimes
\mathfrak{X}(M),[\![{}\cdot{},{}\cdot{}]\!])$ with relations summarized in the
following table:
\begin{equation}%
\begin{tabular}
[c]{cccc}%
$\lbrack\![{}\cdot{},{}\cdot{}]\!]$ & $P^{C}$ & $P^{V}$ & $R$\\\cline{2-4}%
$P^{C}$ & \multicolumn{1}{|c}{$2R$} & $-2R$ & $0$\\
$P^{V}$ & \multicolumn{1}{|c}{$-2R$} & $2R$ & $0$\\
$R$ & \multicolumn{1}{|c}{$0$} & $0$ & $0$%
\end{tabular}
\ \ \ \ \ . \label{Table}%
\end{equation}
Relations $[\![P^{C},R]\!]=[\![P^{V},R]\!]=0$ are the \emph{Bianchi
identities}. Table (\ref{Table}) implies that $d^{C}$, $d^{V}$, $i_{R}$ and
$L_{R}$ generate a Lie algebra of derivations of $\Lambda(M)$:
\begin{equation}%
\begin{tabular}
[c]{ccccc}%
$\![{}\cdot{},{}\cdot{}]$ & $d^{C}$ & $d^{V}$ & $i_{R}$ & $L_{R}$\\\cline{2-5}%
$d^{C}$ & \multicolumn{1}{|c}{$2L_{R}$} & $-2L_{R}$ & $L_{R}$ & $0$\\
$d^{V}$ & \multicolumn{1}{|c}{$-2L_{R}$} & $2L_{R}$ & $0$ & $0$\\
$i_{R}$ & \multicolumn{1}{|c}{$L_{R}$} & $0$ & $0$ & $0$\\
$L_{R}$ & \multicolumn{1}{|c}{$0$} & $0$ & $0$ & $0$%
\end{tabular}
\ \ \ \ \ \label{Table'}%
\end{equation}

\begin{lemma}
\label{Lemmadbar} Let $\lambda\in\overline{\Lambda}$. Then
\[
\overline{d}\lambda=d^{C}\lambda-i_{R}\lambda.
\]

\end{lemma}

\begin{proof}
Both $\overline{d}$ and $d^{C}-i_{R}$ are $\Lambda(M)$-valued derivations of
$\overline{\Lambda}$. They coincide provided they coincide on functions and
$\overline{d}$-exact elements in $\overline{\Lambda}{}^{1}$. For $f\in
C^{\infty}(M)$,
\[
(d^{C}-i_{R})f=d^{C}\!f=\overline{d}f.
\]
Moreover,
\[
(d^{C}-i_{R})\overline{d}f=((d^{C})^{2}-i_{R}d^{C})f=L_{R}f-[i_{R}%
,d^{C}]f=0=\overline{d}\overline{d}f
\]

\end{proof}

\begin{corollary}
Let $\lambda\in\overline{\Lambda}$. Then
\[
\overline{d}\lambda=\overline{d^{C}\lambda}.
\]

\end{corollary}

\begin{proof}
It immediately follows from the above lemma and the obvious fact that
$\overline{i_{R}\lambda}=0$.
\end{proof}

\begin{lemma}
\label{Lem}Let $Z\in\overline{\Lambda}\otimes\overline{\mathfrak{X}}$. Then
\[
\overline{d}Z=[\![P^{C},Z]\!]-[R,Z]_{\mathrm{nr}}.
\]

\end{lemma}

\begin{proof}
Interpret both $\overline{d}$ and $\delta:=[\![P^{C},{}\cdot{}]\!]-[R,{}%
\cdot{}]_{\mathrm{nr}}$ as operators from $\overline{\Lambda}\otimes
\overline{\mathfrak{X}}$ to $\Lambda\otimes\mathfrak{X}$. If $\Delta=d$ or
$\delta$, then
\[
\Delta(\lambda Z)=(\overline{d}\lambda)Z+(-)^{\bar{\omega}}\lambda\Delta(Z).
\]
Therefore, $\overline{d}$ and $\delta$ coincide provided they coincide on zero
degree elements. Thus, let $Y\in\mathfrak{X}(M)$ and $X\in C\mathfrak{X}$.
Then $\overline{Y}=[Y,P^{V}]_{\mathrm{nr}}$ and
\begin{align*}
i_{X}\overline{d}\overline{Y}  &  =(\overline{d}\overline{Y})(X)\\
&  =\overline{[X,\overline{Y}]}\\
&  =-[[\![Y,X]\!],P^{V}]_{\mathrm{nr}}\\
&  =-[\![\overline{Y},[X,P^{V}]_{\mathrm{nr}}]\!]+[X,[\![\overline{Y}%
,P^{V}]\!]]_{\mathrm{nr}}\\
&  =i_{X}[\![P^{C},\overline{Y}]\!]\\
&  =i_{X}([\![P^{C},\overline{Y}]\!]-[R,\overline{Y}]_{\mathrm{nr}})\\
&  =i_{X}\delta\overline{Y},
\end{align*}
where I used Formula \ref{f1}. It follows from arbitrariness of $Y$ and $X$ that $\overline{d} = \delta$.
\end{proof}

\begin{corollary}
Let $Z\in\overline{\Lambda}\otimes\overline{\mathfrak{X}}$. Then
\[
\overline{d}Z=\overline{[\![P^{C},Z]\!]}.
\]

\end{corollary}

\begin{proof}
It immediately follows from the above lemma and the fact that $\overline
{[R,Z]_{\mathrm{nr}}}=0$.
\end{proof}

Notice that, a priori, $\overline{d}$ is defined on longitudinal differential
forms only. However, I need to extend it to the whole $\Lambda(M)$. This can
be done in two ways, both exploiting the transversal distribution $V$. Namely,
consider the derivation $d^{C} - i_{R} : \Lambda(M) \longrightarrow\Lambda
(M)$. In view of Lemma \ref{Lemmadbar}, it extends $\overline{d}$.
Alternatively, identify $\Lambda(M)$ and $\overline{\Lambda} \otimes C\Lambda$
and extend $\overline{d}$ to it as in Remark \ref{RemExt}. It is easy to see
that, actually, these two extensions coincide. Indeed, $C\Lambda$ is generated
by differential $1$-forms of the kind $P^{V} f$, $f \in C^{\infty}(M)$. Let $Y
\in\overline{\mathfrak{X}}$, then
\[
i_{Y}\overline{d}P^{V}f=i_{\overline{d}Y}P^{V}f-\overline{d}i_{Y}%
P^{V}f=i_{[\![P^{C},Y]\!]}P^{V}f-d^{C}i_{Y}P^{V}f=i_{Y}(d^{C}-i_{R})P^{V}f.
\]
It follows from the arbitrariness of $f$ and $Y$, that $\overline{d} = d^{C} -
i_{R}$.

\section{The Homotopy Lie-Rinehart Algebra of a Foliation\label{SHLRF}}

In the following
put $A:=\overline{\Lambda}$ and $\mathcal{Q}:=\overline{\Lambda}{}%
\otimes\overline{\mathfrak{X}}[1]$. According to Remark \ref{15}, one can
identify $\mathrm{Sym}_{A}(\mathcal{Q},A)$ with $\overline{\Lambda}\otimes
C\Lambda=\Lambda(M)$ in such a way that the product (\ref{Prod}) identifies
with the exterior product of differential forms. In particular $\mathrm{Sym}%
_{A}^{r}(\mathcal{Q},A)^{s}$ identifies with $\overline{\Lambda}{}^{s}\otimes
C\Lambda^{r}$. In the following, I will assume this identification. For
$\omega\in\overline{\Lambda}\otimes C\Lambda^{k}=\mathrm{Sym}_{A}%
^{k}(\mathcal{Q},A)$, I denote by
\[
\langle\omega|Z_{1},\ldots,Z_{k}\rangle\in A
\]
its action on elements $Z_{1},\ldots Z_{k}\in\mathcal{Q}$, so not to cause
confusion with the action of differential forms on vector fields. It is easy
to see that
\begin{equation}
\langle\omega|Z_{1},\ldots,Z_{k}\rangle=(-)^{\chi}i_{Z_{1}}\cdots i_{Z_{k}%
}\omega,\quad\text{where }\chi=r+\bar{\omega}\left(  \tfrac{r(r-1)}{2}+%
{\textstyle\sum\nolimits_{i=1}^{k}}
\bar{Z}_{i}\right)  . \label{26}%
\end{equation}

In view of Corollary \ref{Cor}, the existence of the de Rham differential
$d:\mathrm{Sym}_{A}(\mathcal{Q},A)\longrightarrow\mathrm{Sym}_{A}%
(\mathcal{Q},A)$ by itself implies the existence of an $LR_{\infty}%
[1]$-algebra structure on $\mathcal{Q}$. Denote anchors and brackets as usual.
I want to describe them. First of all, notice that $d$ decomposes as
\[
d=d_{0}+d_{1}+d_{2}%
\]
where
\begin{align*}
d_{0}  &  =\overline{d}=d^{C}-i_{R}\\
d_{1}  &  =d^{V}+2i_{R}\\
d_{2}  &  =-i_{R},
\end{align*}
$d_{k}$ is of bidegree $(k,-k+1)$, $k=0,1,2$, and $R$ is the curvature of the
splitting $V$ (see previous section). Moreover, from $d^{2}=0$ it follows that
$\sum_{i+j=k}[d_{i},d_{j}]=0$, $k=0,\ldots,4$, i.e.,
\[
\lbrack d_{0},d_0]=[d_{0},d_{1}]=[d_{1},d_{2}]=[d_{2},d_{2}]=0,
\]
and
\[
\lbrack d_{1},d_{1}]=-2[d_{0},d_{2}]=2L_{R},
\]
where I also used Table (\ref{Table'}). The pair $(\Lambda(M);\{d_{0}%
,d_{1},d_{2}\})$ is the \emph{Maurer-Cartan algebra} of a foliation first
discussed by Huebschmann in \cite{h05}. In Remark 4.16 of \cite{h05},
Huebschmann mentions that \textquotedblleft\emph{P.{} Michor pointed out
[\ldots] a possible connection of the notion of quasi-Lie-Rinehart bracket
with that of Fr\"{o}licher-Nijenhuis bracket}\textquotedblright. Such a
connection exists indeed, as shown by the following theorem, which provides a
description of the $LR_{\infty}[1]$-algebra structure, determined by the de
Rham differential of $M$ via Corollary \ref{Cor}, in terms of
Fr\"{o}licher-Nijenhuis calculus.

\begin{theorem}
The pair $(A,\mathcal{Q})$ possesses the structure of an $LR_{\infty}%
[1]$-algebra, such that
\begin{align*}
\{|\lambda\}  &  =\overline{d}\lambda\\
\{Z|\lambda\}  &  =-(-)^{Z}L_{Z}\lambda+i_{[R,Z]_{\mathrm{nr}}}\lambda\\
\{Z_{1},Z_{2}|\lambda\}  &  =-i_{[[R,Z_{1}]_{\mathrm{nr}},Z_{2}]_{\mathrm{nr}%
}}\lambda\\
\{Z_{1},\ldots,Z_{k-1}|\lambda\}  &  =0\quad\text{for }k>3
\end{align*}
and
\begin{align*}
\{Z\}  &  =\overline{d}Z\\
\{Z_{1},Z_{2}\}  &  =-(-)^{Z_{1}}[\![Z_{1},Z_{2}]\!]+[[R,Z_{1}]_{\mathrm{nr}%
},Z_{2}]_{\mathrm{nr}}\\
\{Z_{1},Z_{2},Z_{3}\}  &  =-[[[R,Z_{1}]_{\mathrm{nr}},Z_{2}]_{\mathrm{nr}%
},Z_{3}]_{\mathrm{nr}}\\
\{Z_{1},\ldots,Z_{k}\}  &  =0\quad\text{for }k>3
\end{align*}
for all $\lambda\in A$, $Z,,Z_{i}\in\mathcal{Q}$, $i=1,2,3,\ldots,k$.
\end{theorem}

\begin{proof}
First of all notice that the right hand sides of all identities in the
statement of the theorem are $\mathbb{R}$-multilinear and graded symmetric in
the $Z_{i}$'s. Therefore, I can apply the standard trick and prove the
identities just for $Z_{i}=Z$ even. Compute the anchors: let $\lambda\in A$,%
\begin{align*}
\{Z|\lambda\}  &  =(d_{1}\lambda)(Z)\\
&  =-i_{Z}(d^{V}+2i_{R})\lambda\\
&  =-[i_{Z},d^{V}+2i_{R}]\lambda\\
&  =-(L_{Z}+i_{[\![P^{C},Z]\!]}+2i_{[Z,R]_{\mathrm{nr}}})\lambda\\
&  =-(L_{Z}-i_{[R,Z]_{\mathrm{nr}}}+i_{\overline{d}Z})\lambda\\
&  =-(L_{Z}-i_{[R,Z]_{\mathrm{nr}}})\lambda,
\end{align*}
where I used Lemma \ref{Lem} and the fact the $i_{Z}\lambda=0$ for all $Z\in$
$\mathcal{Q}$ and $\lambda\in A$.

Similarly,
\begin{align*}
\{Z,Z|\lambda\}  &  =(d_{2}\lambda)(Z,Z)\\
&  =-i_{Z}i_{Z}i_{R}\lambda\\
&  =-i_{i_{Z}i_{Z}R}\lambda\\
&  =-i_{[[R,Z]_{\mathrm{nr}},Z]_{\mathrm{nr}}}\lambda.
\end{align*}
Since $d_{k}=0$ for $k>2$, higher anchors vanish.

Now compute the brackets. Let $\omega\in\mathcal{Q}^{\ast}=\overline{\Lambda
}{}{}{}\otimes C\Lambda^{1}$. Then
\begin{align*}
(-)^{\omega}\langle\omega|\{Z\}\rangle &  =\{\langle\omega|Z\rangle\}-\langle
d_{1}\omega|Z\rangle\\
&  =-\overline{d}i_{Z}\omega+i_{Z}\overline{d}\omega\\
&  =-[\overline{d},i_{Z}]\omega\\
&  =-i_{\overline{d}Z}\omega\\
&  =(-)^{\omega}\langle\omega|\overline{d}Z\rangle,
\end{align*}
where I used (\ref{26}). Similarly,
\begin{align*}
(-)^{\omega}\langle\omega|\{Z,Z\}\rangle &  =2\{Z|\langle\omega|Z\rangle
\}-\langle d_{1}\omega|Z,Z\rangle\\
&  =2(L_{Z}-i_{[R,Z]_{\mathrm{rn}}})i_{Z}\omega-i_{Z}^{2}(d^{V}+2i_{R}%
)\omega\\
&  =2(L_{Z}-i_{[R,Z]_{\mathrm{rn}}})i_{Z}\omega-[i_{Z}^{2},(d^{V}%
+2i_{R})]\omega\\
&  =2(L_{Z}-i_{[R,Z]_{\mathrm{rn}}})i_{Z}\omega-[i_{Z},(d^{V}+2i_{R}%
)]i_{Z}\omega\\
&  -i_{Z}[i_{Z},(d^{V}+2i_{R})]\omega\\
&  =2(L_{Z}-i_{[R,Z]_{\mathrm{rn}}})i_{Z}\omega-(L_{Z}-i_{[R,Z]_{\mathrm{rn}}%
}+i_{\overline{d}Z})i_{Z}\omega\\
&  -i_{Z}(L_{Z}-i_{[R,Z]_{\mathrm{rn}}}+i_{\overline{d}Z})\omega\\
&  =[L_{Z}-i_{[R,Z]_{\mathrm{rn}}},i_{Z}]\omega\\
&  =i_{[\![Z,Z]\!]-[[R,Z]_{\mathrm{nr}},Z]_{\mathrm{nr}}}\omega\\
&  =(-)^{\omega}\langle\omega|-[\![Z,Z]\!]+[[R,Z]_{\mathrm{nr}}%
,Z]_{\mathrm{nr}}\rangle.
\end{align*}
Finally,
\begin{align*}
(-)^{\omega}\langle\omega|\{Z,Z,Z\}\rangle &  =3\{Z,Z|\langle\omega
|Z\rangle\}-\langle d_{2}\omega|Z,Z,Z\rangle\\
&  =-3i_{[[R,Z]_{\mathrm{nr}},Z]_{\mathrm{nr}}}i_{Z}\omega+i_{Z}^{3}%
i_{R}\omega\\
&  =-3i_{[[R,Z]_{\mathrm{nr}},Z]_{\mathrm{nr}}}i_{Z}\omega+i_{Z}[i_{Z}%
^{2},i_{R}]\omega\\
&  =-3i_{[[R,Z]_{\mathrm{nr}},Z]_{\mathrm{nr}}}i_{Z}\omega-i_{Z}^{2}%
[i_{R},i_{Z}]\omega-i_{Z}[i_{R},i_{Z}]i_{Z}\omega\\
&  =-3i_{[[R,Z]_{\mathrm{nr}},Z]_{\mathrm{nr}}}i_{Z}\omega-[i_{Z}%
^{2},i_{[R,Z]_{\mathrm{nr}}}]\omega-i_{Z}i_{[R,Z]_{\mathrm{nr}}}i_{Z}\omega\\
&  =-3i_{[[R,Z]_{\mathrm{nr}},Z]_{\mathrm{nr}}}i_{Z}\omega-2[i_{Z}%
,i_{[R,Z]_{\mathrm{nr}}}]i_{Z}\omega-i_{Z}[i_{Z},i_{[R,Z]_{\mathrm{nr}}%
}]\omega\\
&  =-i_{[[R,Z]_{\mathrm{nr}},Z]_{\mathrm{nr}}}i_{Z}\omega+i_{Z}%
i_{[[R,Z]_{\mathrm{nr}},Z]_{\mathrm{nr}}}\omega\\
&  =-[i_{[[R,Z]_{\mathrm{nr}},Z]_{\mathrm{nr}}},i_{Z}]\omega\\
&  =-i_{[[[R,Z]_{\mathrm{nr}},Z]_{\mathrm{nr}},Z]_{\mathrm{nr}}}\omega\\
&  =-(-)^{\omega}\langle\omega|[[[R,Z]_{\mathrm{nr}},Z]_{\mathrm{nr}%
},Z]_{\mathrm{nr}}\rangle.
\end{align*}
Similarly as above, higher brackets vanish.
\end{proof}

\section{Change of Splitting\label{CS}}

A priori the $LR_{\infty}[1]$-algebra described in the previous section
depends on the choice of the complementary distribution $V$. In fact, it does
not, up to isomorphisms, as an immediate consequence of its derivation from
the (multi-)differential algebra $(\Lambda(M),d)$. Namely, let $V^{\prime}$ be a
different complementary distribution. Denote by $\overline{\Lambda}{}^{\prime
}$ (resp., $\overline{\mathfrak{X}}{}^{\prime}$) the image of $\overline
{\Lambda}$ (resp., $\overline{\mathfrak{X}}$) under the embedding
$\overline{\Lambda}\longrightarrow\Lambda(M)$ (resp., $\overline{\mathfrak{X}%
}\longrightarrow\mathfrak{X}(M)$) determined by $V^{\prime}$. The algebras
$\overline{\Lambda}{}\otimes C\Lambda^{k}$ and $\overline{\Lambda}{}^{\prime
}\otimes C\Lambda^{k}$ both identify with $\Lambda(M)$ and, in view of
Definition \ref{Def}, $d$ induces isomorphic $LR_{\infty}[1]$-algebra
structures on $(\overline{\Lambda},\mathcal{Q})$ and $(\overline{\Lambda}%
{}^{\prime},\mathcal{Q}^{\prime}:=\overline{\Lambda}{}^{\prime}\otimes
\overline{\mathfrak{X}}{}^{\prime}[1])$ (up to the obvious identifications
$\overline{\mathrm{id}}:\overline{\Lambda}{}^{\prime}\longrightarrow
\overline{\Lambda}$, and $\overline{\mathrm{id}}:\overline{\Lambda}{}^{\prime
}\otimes\overline{\mathfrak{X}}{}^{\prime}\longrightarrow\overline{\Lambda}%
{}\otimes\overline{\mathfrak{X}}{}$, all the $\overline{\mathrm{id}}$'s
appearing below in this section are due to this). Let $\psi:\overline{\Lambda
}\otimes C\Lambda\longrightarrow\overline{\Lambda}{}^{\prime}\otimes C\Lambda$
be the composition of isomorphisms
\[
\overline{\Lambda}\otimes C\Lambda\longrightarrow\Lambda(M)\longrightarrow
\overline{\Lambda}{}^{\prime}\otimes C\Lambda.
\]
Now, I describe the isomorphism $\mathcal{Q}^{\prime}\longrightarrow
\mathcal{Q}$. I will use the same notations as in Section \ref{MSHLR}. Let
${}^{\prime}P^{C}\in\overline{\Lambda}{}^{\prime}\otimes C\mathfrak{X}$ be the
projector on $C$ determined by $V^{\prime}$.

If $V^{\prime}$ is locally spanned by vector fields $\ldots,\partial/\partial
u^{\alpha}+{}^{\prime}V_{\alpha}^{i}\partial_{i},\ldots$, then $^{\prime}%
P^{C}$ is locally given by $^{\prime}P^{C}=\overline{d}{}^{\prime}x^{i}%
\otimes\partial_{i}$, where $\overline{d}{}^{\prime}x^{i}:=dx^{i}-{}^{\prime
}V_{\alpha}^{i}du^{\alpha}$. Put $\Delta:=P^{C}-{}{}^{\prime}P^{C}\in
C\Lambda^{1}\otimes C\mathfrak{X}$. Locally
\[
\Delta=\Delta_{\alpha}^{i}du^{\alpha}\otimes\partial_{i},\quad\Delta_{\alpha
}^{i}:={}^{\prime}V_{\alpha}^{i}-V_{\alpha}^{i}{}.
\]

\begin{proposition}
The maps $\psi_{k}:\overline{\Lambda}\longrightarrow\overline{\Lambda}{}%
{}^{\prime}\otimes C\Lambda^{k}$ and $\Psi_{k}:\overline{\Lambda}{}\otimes
C\Lambda^{1}\longrightarrow\overline{\Lambda}{}{}^{\prime}\otimes C\Lambda
^{k}$ determined by $\psi$ are given by $i_{\Delta}^{k}$ and $i_{\Delta}%
^{k-1}$ respectively.
\end{proposition}

\begin{proof}
The simplest proof is in local coordinates. Thus, let $\lambda\in
\overline{\Lambda}{}$ be locally given by
\[
\lambda=\lambda_{i_{1}\cdots i_{q}}\overline{d}{}x^{i_{1}}\cdots\overline{d}%
{}x^{i_{q}}=\lambda_{i_{1}\cdots i_{q}}(\overline{d}{}^{\prime}x^{i_{1}%
}+\Delta_{\alpha_{1}}^{i_{1}}du^{\alpha_{1}})\cdots(\overline{d}{}^{\prime
}x^{i_{q}}+\Delta_{\alpha_{q}}^{i_{q}}du^{\alpha_{q}}).
\]
Its component in $\overline{\Lambda}{}^{\prime}\otimes C\Lambda^{k}$ is
\[
\psi_{k}(\lambda)=(-)^{kq}k!\tbinom{q}{k}\Delta_{\alpha_{1}}^{i_{1}}%
\cdots\Delta_{\alpha_{k}}^{i_{k}}\lambda_{i_{1}\cdots i_{q}}^{i_{k+1}%
}\overline{d}{}^{\prime}x^{i_{1}}\cdots\overline{d}{}^{\prime}x^{i_{q}}\otimes
du^{\alpha_{1}}\cdots du^{\alpha_{k}}=i_{\Delta}^{k}\lambda.
\]
Similarly, let $\omega\in\overline{\Lambda}{}\otimes C\Lambda^{1}$ be locally
given by
\[
\omega=\omega_{\alpha}\otimes du^{\alpha},\quad\omega_{\alpha}\in
\overline{\Lambda}{}.
\]
Then
\[
\Psi_{k}(\omega)=\psi_{k-1}(\omega_{\alpha})\otimes du^{\alpha}=i_{\Delta
}^{k-1}\omega_{\alpha}\otimes du^{\alpha}=i_{\Delta}^{k-1}\omega.
\]

\end{proof}

Now I will describe the maps $\phi_{k}:(\mathcal{Q}^{\prime})^{\times k}%
\times\overline{\Lambda}{}{}\longrightarrow\overline{\Lambda}$ and $\Phi
_{k}:(\mathcal{Q}^{\prime})^{\times k}\longrightarrow\mathcal{Q}$. For
$Z^{\prime}\in\mathcal{Q}^{\prime}$, it is convenient to put
\[
\Delta Z^{\prime}:=-i_{\Delta}Z^{\prime}.
\]
Now, Let $Z_{1}^{\prime},\ldots,Z_{k}^{\prime}\in\mathcal{Q}^{\prime}$ and
$\lambda\in\overline{\Lambda}{}.$ Clearly,%
\begin{align*}
\phi_{k}(Z_{1}^{\prime},\ldots,Z_{k}^{\prime}|\lambda)  &  =(-)^{\lambda
^{\prime}(Z_{1}^{\prime}+\cdots+Z_{k}^{\prime})}\overline{\mathrm{id}}%
\langle\psi_{k}(\lambda^{\prime})|Z_{1}^{\prime},\ldots,Z_{k}^{\prime}%
\rangle\\
&  =(-)^{\chi}\overline{\mathrm{id}}\left(  i_{Z_{1}^{\prime}}\cdots
i_{Z_{k}^{\prime}}i_{\Delta}^{k}\lambda\right)
\end{align*}
with
\[
\chi=r+\bar{\lambda}\left(  \tfrac{r(r-1)}{2}+%
{\textstyle\sum\nolimits_{i=1}^{k}}
\bar{Z}_{i}^{\prime}\right)  .
\]

\begin{proposition}
$\Phi_{1}(Z^{\prime})=\overline{\mathrm{id}}(Z^{\prime})$ and%
\begin{equation}
\Phi_{k}(Z_{1}^{\prime},\ldots,Z_{k}^{\prime})=\overline{\mathrm{id}}%
\sum_{\lambda\in S_{k}}\alpha(\sigma,\boldsymbol{Z}^{\prime})i_{Z_{\sigma
(1)}^{\prime}}i_{\Delta Z_{\sigma(2)}^{\prime}}\cdots i_{\Delta Z_{\sigma
(k-1)}^{\prime}}\Delta Z_{\sigma(k)}^{\prime} \label{16}%
\end{equation}
for all $k>1$ and $Z^{\prime},Z_{1}^{\prime},\ldots,Z_{k}^{\prime}%
\in\mathcal{Q}^{\prime}$.
\end{proposition}

\begin{proof}
Let $\omega\in\mathcal{Q}^{\ast}$ and $Z^{\prime}$ be an even element of
$\mathcal{Q}^{\prime}$. Then%
\begin{align}
i_{\Phi_{k}(Z^{k})}\omega &  =-\langle\omega|\Phi_{k}(Z^{\prime k}%
)\rangle\nonumber\\
&  =-\overline{\mathrm{id}}\langle\Psi_{k}(\omega)|Z^{\prime k}\rangle
+\overline{\mathrm{id}}\sum_{\substack{m_{1}+m_{2}=k\\m_{1},m_{2}>0}%
}\tbinom{k}{m_{1}}\langle\psi_{m_{1}}\langle\omega|\Phi_{m_{2}}(Z^{\prime
m_{2}})\rangle|Z^{\prime m_{1}}\rangle\nonumber\\
&  =-(-)^{k}\overline{\mathrm{id}}\left(  i_{Z^{\prime}}^{k}i_{\Delta}%
^{k-1}\omega\right)  -\overline{\mathrm{id}}\sum_{\substack{m_{1}%
+m_{2}=k\\m_{1},m_{2}>0}}\tbinom{k}{m_{1}}(-)^{m_{1}}i_{Z^{\prime}}^{m_{1}%
}i_{\Delta}^{m_{1}}i_{\Phi_{m_{2}}(Z^{\prime m_{2}})}\omega\nonumber\\
&  =\cdots\\
&  =\overline{\mathrm{id}}\sum_{s=1}^{k}(-)^{k+s}\sum_{\substack{m_{1}%
+\cdots+m_{s}=k\\m_{1},\ldots,m_{s}>0}}\tbinom{k}{m_{1},\cdots,m_{s}%
}i_{Z^{\prime}}^{m_{1}}i_{\Delta}^{m_{1}}\cdots i_{Z^{\prime}}^{m_{s-1}%
}i_{\Delta}^{m_{s-1}}i_{Z^{\prime}}^{m_{s}}i_{\Delta}^{m_{s}-1}\omega.
\label{17}%
\end{align}
In particular, for $k=1$, one gets
\[
i_{\Phi_{1}(Z^{\prime})}\omega=\overline{\mathrm{id}}\left(  i_{Z^{\prime}%
}\omega\right)  =i_{\overline{\mathrm{id}}(Z^{\prime})}\omega\Longrightarrow
\Phi_{1}(Z^{\prime})=\overline{\mathrm{id}}(Z^{\prime}),
\]
which is the base of induction. Notice that, by linearity, $\Phi_{k}(Z^{\prime
k})$ is known provided $i_{\Phi_{k}(Z^{\prime k})}\omega$ is known for all
$\omega\in C\Lambda^{1}$. But in this case $i_{\Delta}\omega=0$, and
(\ref{17}) reduces to
\begin{align*}
i_{\Phi_{k}(Z^{\prime k})}\omega &  =k\overline{\mathrm{id}}\sum_{s=1}%
^{k}(-)^{k+s}\sum_{\substack{m_{1}+\cdots+m_{s-1}=k-1\\m_{1},\ldots,m_{s-1}%
>0}}\tbinom{k-1}{m_{1},\cdots,m_{s-1}}i_{Z^{\prime}}^{m_{1}}i_{\Delta}^{m_{1}%
}\cdots i_{Z^{\prime}}^{m_{s-1}}i_{\Delta}^{m_{s-1}}i_{Z^{\prime}}\omega\\
&  =-k\overline{\mathrm{id}}\sum_{s=1}^{k}(-)^{k+s}\sum_{\substack{m_{1}%
+\cdots+m_{s-1}=k-1\\m_{1},\ldots,m_{s-1}>0}}\tbinom{k-1}{m_{1},\cdots
,m_{s-1}}i_{Z^{\prime}}^{m_{1}}i_{\Delta}^{m_{1}}\cdots i_{Z^{\prime}%
}^{m_{s-1}}i_{\Delta}^{m_{s-1}-1}i_{\Delta Z^{\prime}}\omega.
\end{align*}
For $k>1$ one gets
\begin{align}
i_{\Phi_{k}(Z^{\prime k})}\omega &  =k\sum_{s=1}^{k-1}(-)^{k-1+s}%
\sum_{\substack{m_{1}+\cdots+m_{s}=k-1\\m_{1},\ldots,m_{s}>0}}\tbinom
{k-1}{m_{1},\cdots,m_{s}}i_{Z^{\prime}}^{m_{1}}i_{\Delta}^{m_{1}}\cdots
i_{Z^{\prime}}^{m_{s}}i_{\Delta}^{m_{s}-1}i_{\Delta Z^{\prime}}\omega
\nonumber\\
&  =k\overline{\mathrm{id}}\left(  i_{\Phi_{k-1}(Z^{\prime k-1})}i_{\Delta
Z^{\prime}}\omega\right)  . \label{18}%
\end{align}
Now let $Z^{\prime}$ be locally given by $Z^{\prime}=Z^{\alpha}\otimes
V_{\alpha}^{\prime}$, $Z^{\alpha}\in\overline{\Lambda}{}{}^{\prime}$, so that
$\Delta Z$ is locally given by $\Delta Z=W_{\beta}^{\alpha}du^{\beta}\otimes
V_{\alpha}^{\prime}$, where $W_{\beta}^{\alpha}:=\Delta_{\beta}^{j}%
i(\partial_{j})Z^{\alpha}$. Similarly, let $\Phi_{\ell}(Z^{\prime\ell})$ be
locally given by $\Phi_{\ell}(Z^{\prime\ell})=\Phi_{\ell}^{\alpha}\otimes
V_{\alpha}^{\prime}$, $\Phi_{\ell}^{\alpha}\in\overline{\Lambda}{}{}$. It
follows from (\ref{18}) that
\[
\Phi_{k}^{\alpha}=k\Phi_{k-1}^{\beta}\overline{\mathrm{id}}(W_{\beta}^{\alpha
})=\cdots=k!\Phi_{1}^{\alpha_{1}}\overline{\mathrm{id}}(W_{\alpha_{1}}%
^{\alpha_{2}}W_{\alpha_{2}}^{\alpha_{3}}\cdots W_{\alpha_{k-2}}^{\alpha_{k-1}%
}W_{\alpha_{k-1}}^{\alpha})=k!\overline{\mathrm{id}}(Z^{\alpha_{1}}%
W_{\alpha_{1}}^{\alpha_{2}}W_{\alpha_{2}}^{\alpha_{3}}\cdots W_{\alpha_{k-2}%
}^{\alpha_{k-1}}W_{\alpha_{k-1}}^{\alpha})
\]
so that
\[
\Phi_{k}(Z^{\prime k})=k!\,\overline{\mathrm{id}}\left(  i_{Z^{\prime}%
}i_{\Delta Z^{\prime}}^{k-2}\Delta Z^{\prime}\right)  .
\]

\end{proof}

\section{On the Homotopy Lie-Rinehart Algebra of a Presymplectic
Manifold\label{SHLRPSF}}

The key idea behind secondary calculus \cite{v98,v01} is to intepret
characteristic cohomologies of an involutive distribution as geometric
structures on the space $\boldsymbol{P}$ of integral manifolds. In Section
\ref{HAF}, I provided two examples of this, namely $\boldsymbol{C}_{%
\mathscr{F}%
}^{\infty}:=H(\overline{\Lambda},d)$ and $\boldsymbol{\mathfrak{X}}_{%
\mathscr{F}%
}:=H(\overline{\Lambda}\otimes\overline{\mathfrak{X}},\overline{d})$. Within
secondary calculus, they are interpreted as functions and vector fields on
$\boldsymbol{P}$, respectively. As I have already remarked, Theorem
\ref{Theorem1} supports this interpretation. I will now discuss more
supporting facts. {Recall that a \emph{presymplectic manifold} $(M,\Omega)$ is
a smooth manifold together with a constant rank, closed $2$-form $\Omega$.
Typical examples of presymplectic manifolds come from symplectic geometry.
Namely, let $M$ be a submanifold in a symplectic manifold $(N,\omega)$. Then,
$M$ together with the restricted $2$-form $\omega|_{N}$, is (almost
everywhere, locally) a presymplectic manifold. Thus, let $(M,\Omega)$ be a
presymplectic manifold}, let $C$ be its characteristic {(involutive)}
distribution, i.e. a vector field $X$ is in $C\mathfrak{X}$ if $i_{X}\Omega
=0$, and $%
\mathscr{F}%
$ its integral foliation. The two form $\Omega$ is naturally interpreted as if
it were a genuine symplectic structure on the space $\tilde{\boldsymbol{P}}$
of leaves of $C$
(for instance, when $\tilde{\boldsymbol{P}}$ is a manifold and the projection
$\pi:M\longrightarrow\tilde{\boldsymbol{P}}$ is a submersion, then
$\Omega:=\pi^{\ast}\Omega_{0}$ for a unique symplectic form on $\tilde
{\boldsymbol{P}}$). This statement can be given the more precise formulation
of Theorem \ref{Theorem2} below. Before stating it, I give some definitions.
First of all notice that, by definition, $\Omega\in C\Lambda^{2}$. Moreover,
it follows from $d\Omega=0$, that
\[
\overline{d}\Omega=d_{2}\Omega=d_{3}\Omega=0.
\]
Now, as above, chose a distribution $V$ which is complementary to $C$. There
is a unique bivector $P\in\Lambda^{2}\overline{\mathfrak{X}}$
\textquotedblleft inverting $\omega$ on $\overline{\mathfrak{X}}%
$\textquotedblright. Clearly, $\overline{d}P=0$. However, as discussed in
\cite{op05}, $P$ is Poisson iff $R=0$, i.e., $V$ is involutive as well.
Nonetheless, it defines an isomorhism
\[
\sharp:\overline{\Lambda}\otimes C\Lambda^{1}\longrightarrow\overline{\Lambda
}\otimes\overline{\mathfrak{X}}%
\]
of $\overline{\Lambda}$-modules in an obvious way. For $\omega_{1},\omega
_{2}\in\overline{\Lambda}\otimes C\Lambda^{1}$, put
\[
\langle\omega_{1}|\omega_{2}\rangle_{\Omega}:=\langle\omega_{1}|\sharp
(\omega_{2})\rangle.
\]

\begin{theorem}
\label{Theorem2}
\end{theorem}

\begin{enumerate}
\item \label{1.2}Cohomology $\boldsymbol{C}_{%
\mathscr{F}%
}^{\infty}$ possesses a canonical structure of graded Poisson algebra
$\boldsymbol{\{}{}\cdot{},{}\cdot{}\boldsymbol{\}}$ given by%
\[
\boldsymbol{\{}{}[\lambda]{},[\lambda^{\prime}]\boldsymbol{\}}:=[\langle
d_{1}\lambda|d_{1}\lambda^{\prime}\rangle_{\Omega}]\in\boldsymbol{\mathfrak{X}%
_{%
\mathscr{F}%
}},\quad\lbrack\lambda],[\lambda^{\prime}]\in\boldsymbol{C}_{%
\mathscr{F}%
}^{\infty},\text{\quad}\lambda,\lambda^{\prime}\in\overline{\Lambda}.
\]
The brackets $\boldsymbol{\{}{}\cdot{},{}\cdot{}\boldsymbol{\}}$ is
independent of the choice of $V$.

\item \label{2.2}There is a canonical morphism of graded Lie algebras
$\boldsymbol{X}:(\boldsymbol{C}_{%
\mathscr{F}%
}^{\infty},\boldsymbol{\{}{}\cdot{},{}\cdot{}\boldsymbol{\}})\longrightarrow
(\boldsymbol{\mathfrak{X}}_{%
\mathscr{F}%
},\boldsymbol{[}\cdot{},{}\cdot{}\boldsymbol{]})$ given by
\[
\boldsymbol{X}:\boldsymbol{C}_{%
\mathscr{F}%
}^{\infty}\ni\lbrack\lambda]\longmapsto\lbrack\sharp d_{1}\lambda
]\in\boldsymbol{\mathfrak{X}}_{%
\mathscr{F}%
},\quad\lbrack\lambda]\in\boldsymbol{C}_{%
\mathscr{F}%
}^{\infty},\text{\quad}\lambda\in\overline{\Lambda}.
\]
The morphism $\boldsymbol{X}_{%
\mathscr{F}%
}$ is independent of the choice of $V$.
\end{enumerate}

The graded Poisson algebra of Theorem \ref{Theorem2} actually comes from an
$L_{\infty}[1]$-algebra structure $\mathscr{P}$ in $\overline{\Lambda}$, and
the morphism $\boldsymbol{X}$ comes from a morphism of $L_{\infty}%
[1]$-algebras $(\overline{\Lambda},\mathscr{P})\longrightarrow(\mathcal{Q}%
,\mathscr{Q})$ (here, $\mathscr{Q}$ is the canonical $L_{\infty}[1]$-algebra
structure on $\mathcal{Q}=\overline{\Lambda}\otimes\overline{\mathfrak{X}}%
[1]$), according to the following

\begin{theorem}
\label{Theorem3}

\begin{enumerate}
\item \label{1.3}The vector space $\overline{\Lambda}[1]$ possesses a
structure of $L_{\infty}[1]$-algebra $\mathscr{L}=\{\{{}\cdot{},\cdots,{}%
\cdot{}\}_{k}^{\mathrm{op}},\ k\in\mathbb{N}\}$ such that $\{{}{}\cdot{}%
\}_{1}^{\mathrm{op}}=\overline{d}$,

\item \label{2.3}There is a morphism of $L_{\infty}[1]$-algebras
$X:(\overline{\Lambda},\mathscr{P})\longrightarrow(\mathcal{Q},\mathscr{Q})$.
\end{enumerate}

Moreover, the structure in \ref{1.3} (resp., the morphism in \ref{2.3}),
induce the structure in \ref{1.2} (resp., the morphism in \ref{2.2}) of
Theorem \ref{Theorem2} in cohomology, up to a sign (due to the chosen sign conventions).
\end{theorem}

Part \ref{1.3}{} of Theorem \ref{Theorem3} has been proved by Oh and Park in
\cite{op05} (which motivates the notation for the brackets in $\mathscr{L}$).
In the remaining part of this section, I prove Part \ref{2.3}. First, I recall
the definition of $\mathscr{L}$ \cite{op05}. Let $R^{\sharp}$ be the tensor
obtained by contracting one lower index of $R$ with one upper index of $P$.
Interpret $R^{\sharp}$ as a $\mathrm{End\,}C\Lambda^{1}$-valued derivation of
$C^{\infty}(M)$: $R^{\sharp}\in\mathrm{End\,}C\Lambda^{1}\otimes
C\mathfrak{X}$. If $\lambda\in\overline{\Lambda}[1]$ I will consider
\[
i_{R^{\sharp}}\lambda\in\overline{\Lambda}\otimes\mathrm{End\,}C\Lambda
^{1}\simeq\mathrm{End}_{\overline{\Lambda}}\,(\overline{\Lambda}\otimes
C\Lambda^{1}).
\]
Then
\begin{equation}
\{\lambda_{1},\ldots,\lambda_{k}\}_{k}^{\mathrm{op}}:=\sum_{\sigma\in S_{k}%
}\alpha(\sigma,\boldsymbol{\lambda})\langle d_{1}\lambda_{\sigma
(1)}|(i_{R^{\sharp}}\lambda_{\sigma(2)}\circ\cdots\circ i_{R^{\sharp}}%
\lambda_{\sigma(k-1)})(d_{1}\lambda_{\sigma(k)})\rangle_{\Omega} \label{19}%
\end{equation}
For all $\lambda_{1},\ldots,\lambda_{k}\in\overline{\Lambda}[1]$.

Now, I define the morphism $X:(\overline{\Lambda},\mathscr{P})\longrightarrow
(\mathcal{Q},\mathscr{Q})$. It is a \emph{homotopy version} of the standard
morphism sending Hamiltonians to their Hamiltonian vector fields on a
symplectic manifold {and, to my knolewdge, it is defined here for the first
time}. Define maps
\[
X_{k}:\overline{\Lambda}[1]^{\times k}\longrightarrow\mathcal{Q}%
\]
via
\[
X_{k}(\lambda_{1},\ldots,\lambda_{k})(f):=\{\lambda_{1},\ldots,\lambda
_{k},f\}_{k+1}^{\mathrm{op}}%
\]
It follows from (\ref{19}) that $X_{k}(\lambda_{1},\ldots,\lambda_{k}%
)\in\overline{\Lambda}\otimes\overline{\mathfrak{X}}$ so that $X_{k}$ is a
well defined degree $0$ map for all $k$.

\begin{lemma}
\label{21}Let $\lambda\in\overline{\Lambda}[1]$ be even. Put $Z_{k}%
:=X_{k}(\lambda^{k})$. Then%
\begin{equation}
\{Z_{k}|\lambda^{\prime}\}=\{\lambda^{k},\lambda^{\prime}\}^{\mathrm{op}}%
-\sum_{\substack{i+j=k\\i,j>0}}\tbinom{i+j}{i}i_{Z_{i}}i_{Z_{j}}i_{R}%
\lambda^{\prime} \label{22}%
\end{equation}
for all $\lambda^{\prime}\in\overline{\Lambda}$.
\end{lemma}

\begin{proof}
Both hand sides of (\ref{22}) are derivations in the argument $\lambda
^{\prime}$. Therefore, it is enough to check (\ref{22}) on generators of
$\overline{\Lambda}$, i.e., for $\lambda^{\prime}=f$ and $\overline{d}f$ for
$f\in C^{\infty}(M)$. When $\lambda^{\prime}=f$, (\ref{22}) is trivially true
by definition of $Z_{k}$. Now, it easily follows from $\overline{d}P=0$ and
the Bianchi identities that
\[
\{Z_{k}|\overline{d}f\}=k!\langle d_{1}\overline{d}f|(i_{R^{\sharp}}%
\lambda)^{k-1}(d_{1}\lambda)\rangle_{\Omega}+k!\langle d_{1}\lambda
|(i_{R^{\sharp}}\lambda)^{k-1}(d_{1}\overline{d}f)\rangle_{\Omega}.
\]
On the other hand%
\begin{align*}
\{\lambda^{k},\overline{d}f\}^{\mathrm{op}}  &  =k!\langle d_{1}\overline
{d}f|(i_{R^{\sharp}}\lambda)^{k-1}(d_{1}\lambda)\rangle_{\Omega}+k!\langle
d_{1}\lambda|(i_{R^{\sharp}}\lambda)^{k-1}(d_{1}\overline{d}f)\rangle_{\Omega
}\\
&  \quad\ +\sum_{r+s=k-2}\tbinom{r+s}{r}\langle d_{1}\lambda|(i_{R^{\sharp}%
}\lambda)^{r}\circ(i_{R^{\sharp}}\overline{d}f)\circ(i_{R^{\sharp}}%
\lambda)^{s}(d_{1}\lambda)\rangle_{\Omega}\\
&  =\{Z_{k}|\overline{d}f\}+\sum_{\substack{i+j=k\\i,j>0}}\tbinom{i+j}%
{i}i_{Z_{i}}i_{Z_{j}}R(f),
\end{align*}
where the final equality can be easily checked using, for instance, local coordinates.
\end{proof}

\begin{proposition}
$X:=\{X_{k},\ k\in\mathbb{N}\}$ is a morphism of $L_{\infty}[1]$-algebras.
\end{proposition}

\begin{proof}
Let $\lambda\in\overline{\Lambda}[1]$ be an even element. I will prove that
\[
K_{X}^{m}(\lambda^{m})=0
\]
for all such $\lambda$. Now, for $f\in C^{\infty}(M)$,
\begin{align*}
K_{X}^{m}(\lambda^{m})(f)  &  =\sum_{j+k=m}\tbinom{j+k}{j}X_{k+1}%
(\{\lambda^{j}\}^{\mathrm{op}},\lambda^{k})(f)\\
&  \quad\ -\sum_{r=1}^{m}\sum_{\substack{k_{1}+\cdots+k_{r}=m\\0<k_{1}%
\leq\cdots\leq k_{r}}}C(k_{1},\ldots,k_{r})\{X_{k_{1}}(\lambda^{k_{1}}%
),\ldots,X_{k_{r}}(\lambda^{k_{r}})\}(f),
\end{align*}
where
\[
X_{k+1}(\{\lambda^{j}\}^{\mathrm{op}},\lambda^{k})(f)=\{\{\lambda
^{j}\}^{\mathrm{op}},\lambda^{k},f\}^{\mathrm{op}}.
\]
Now, put $Z_{k}:=X_{k}(\lambda^{k})$ for all $k$, and compute
\begin{align}
\{Z_{k_{1}},\ldots,Z_{k_{r}}\}(f)  &  =\{\{Z_{k_{1}},\ldots,Z_{k_{r}%
}\}|f\}\nonumber\\
&  =-\sum_{s+t=r}\sum_{\sigma\in S_{s,t}}\{\{Z_{k_{\sigma(1)}},\ldots
,Z_{k_{\sigma(s)}}|f\},Z_{k_{\sigma(s+1)}},\ldots Z_{k_{\sigma(s+t)}%
}\}^{\oplus}\nonumber\\
&  =-\sum_{s+t=r}\sum_{\sigma\in S_{s,t}}\{Z_{k_{\sigma(s+1)}},\ldots
Z_{k_{\sigma(s+t)}}|\{Z_{k_{\sigma(1)}},\ldots,Z_{k_{\sigma(s)}}|f\}\},
\label{20}%
\end{align}
while the highest anchor vanishes on functions and does not contribute. For
the same reason only summands with $s=0,1$ survive in (\ref{20}) and one gets
\[
\{Z_{k_{1}},\ldots,Z_{k_{r}}\}(f)=-\{Z_{k_{1}},\ldots Z_{k_{r}}|\{f\}\}-\sum
_{i=1}^{r}\{Z_{k_{1}},\ldots,\widehat{Z_{k_{i}}},\ldots,Z_{k_{r}}|\{Z_{k_{i}%
}|f\}\},
\]
which is non-zero only when $r=1,2,3$. In view of Lemma \ref{21},%
\begin{align*}
\{Z_{m}\}(f)  &  =-\{Z_{m}|\overline{d}f\}-\overline{d}\{Z_{m}|f\}\\
&  =-\{\lambda^{m},\{f\}^{\mathrm{op}}\}^{\mathrm{op}}-\{\{\lambda
^{m},f\}^{\mathrm{op}}\}^{\mathrm{op}}+\sum_{\substack{i+j=m\\i,j>0}%
}\tbinom{i+j}{i}i_{Z_{i}}i_{Z_{j}}i_{R}\overline{d}f,
\end{align*}%
\begin{align*}
\{Z_{j},Z_{k}\}(f)  &  =-\dfrac{1}{2}\{Z_{j},Z_{k}|\{f\}\}-\{Z_{j}%
|\{Z_{k}|f\}\}+{}\overset{j,k}{\leftrightarrow}\\
&  =-\dfrac{1}{2}i_{Z_{j}}i_{Z_{k}}i_{R}\overline{d}f-\{\lambda^{j}%
,\{\lambda^{k},f\}^{\mathrm{op}}\}^{\mathrm{op}}+\sum_{\substack{r+s=j\\r,s>0}%
}\tbinom{r+s}{r}i_{Z_{r}}i_{Z_{s}}i_{R}\{Z_{k}|f\}+{}\overset{j,k}%
{\leftrightarrow}%
\end{align*}
and%
\[
\{Z_{j},Z_{k},Z_{\ell}\}(f)=-\{Z_{j},Z_{k}|\{Z_{\ell}|f\}\}+{}\overset{j,k,
\ell}{\circlearrowleft}{}=-i_{Z_{j}}i_{Z_{k}}i_{R}\{Z_{\ell}|f\}+{}%
\overset{j,k, \ell}{\circlearrowleft}{}.
\]

I conclude that
\begin{align*}
K_{X}^{m}(\lambda^{m})(f)  &  =\sum_{j+k=m}\tbinom{j+k}{j}\{\{\lambda
^{j}\}^{\mathrm{op}},\lambda^{k},f\}^{\mathrm{op}}+\{\lambda^{m}%
,\{f\}^{\mathrm{op}}\}^{\mathrm{op}}+\{\{\lambda^{m},f\}^{\mathrm{op}%
}\}^{\mathrm{op}}\\
&  \quad\ +\sum_{\substack{j+k=m\\j,k>0}}\tbinom{j+k}{j}\{\lambda
^{j},\{\lambda^{k},f\}^{\mathrm{op}}\}^{\mathrm{op}},
\end{align*}
where the right hand side is the $(m+1)$st Jacobiator of $\mathscr{L}$.
\end{proof}

\part{Appendixes}

\appendix{}

\section{Proofs of Theorems \ref{23} and \ref{TheorMor}: Computational
Details\label{Appendix}}

For completeness, in this appendix, I add the (mostly straightforward)
computational details of the proofs of Theorems \ref{23} and \ref{TheorMor}.
To be clear, I organize these details in lemmas. I use the same notations as
in Sections \ref{SHDGLR} and \ref{MSHLR}.

\begin{lemma}
\label{Alemma1}Let $\omega\in\mathcal{Q}^{\ast}$. For $q_{1},\ldots,q_{k+1}%
\in\mathcal{Q}$,
\[
\ (d_{k}\omega)(q_{1},\ldots,q_{k+1}):=\sum_{i=1}^{k}(-)^{\chi}\{q_{1}%
,\ldots,\widehat{q_{i}},\ldots,q_{k+1}|\omega(q_{i})\}+(-)^{\omega}%
\omega(\{q_{1},\ldots,q_{k+1}\}),
\]
where $\chi:=\bar{\omega}(\bar{q}_{1}+\cdots+\widehat{\bar{q}_{i}}+\cdots
\bar{q}_{k+1})+\bar{q}_{i}(\bar{q}_{i+1}+\cdots+\bar{q}_{k+1})$, and a hat $\widehat{\cdot}$ denotes omission. Then $d_{k}\omega\in\mathrm{Sym}_{A}%
^{k+1}(\mathcal{Q},A)$.
\end{lemma}

\begin{proof}
Clearly, $d_{k}\omega$ is $K$-multilinear and graded symmetric. Now, let
$q_{2}=\cdots=q_{k+1}=q$ be even. Compute
\begin{align*}
(d_{k}\omega)(aq_{1},q^{k})  &  =(-)^{(a+q_{1})\omega}k\{aq_{1},q^{k-1}%
|\omega(q)\}+\{q^{k}|\omega(aq_{1})\}-(-)^{\omega}\omega(\{aq_{1},q^{k}\})\\
&  =(-)^{a(\omega+1)+q_{1}\omega}ka\{q_{1},q^{k-1}|\omega(q)\}+(-)^{a\omega
}\{q^{k}|a\omega(q_{1})\}\\
&  \quad\ -(-)^{a(\omega+1)+\omega}a\omega(\{q_{1},q^{k}\})-(-)^{a\omega
}\{q^{k}|a\}\omega(q_{1}\}\\
&  =(-)^{a(\omega+1)+q_{1}\omega}ka\{q_{1},q^{k-1}|\omega(q)\}+(-)^{a(\omega
+1)}a\{q^{k}|\omega(q_{1})\}\\
&  \quad\ -(-)^{a(\omega+1)+\omega}a\omega(\{q_{1},q^{k}\})\\
&  =(-)^{a(\omega+1)}a(d_{k}\omega)(q_{1},q^{k}).
\end{align*}

\end{proof}

\begin{lemma}
\label{Alemma2}Let $a\in A$ and $\omega\in\mathcal{Q}^{\ast}$. Then
\[
d_{k}(a\omega)=(d_{k}a)\omega+(-)^{a}a(d_{k}\omega).
\]

\end{lemma}

\begin{proof}
Let $q\in\mathcal{Q}$ be even. Then%
\begin{align*}
d_{k}(a\omega)(q^{k+1})  &  =(k+1)\{q^{k}|\,(a\omega)(q)\}-(-)^{a+\omega
}(a\omega)(\{q^{k+1}\})\\
&  =(-)^{a}(k+1)a\{q^{k}|\,\omega(q)\}+(k+1)\{q^{k}|\,a\}\omega
(q)-(-)^{a+\omega}a\omega(\{q^{k+1}\})\\
&  =(-)^{a}a(d_{k}\omega)(q^{k+1})+(k+1)(d_{k}a)(q^{k})\omega(q)\\
&  =((d_{k}a)\omega+(-)^{a}a(d_{k}\omega))(q^{k+1}).
\end{align*}

\end{proof}

\begin{lemma}
\label{Alemma3}The derivation $d_{k}:\mathrm{Sym}_{A}(\mathcal{Q}%
,A)\longrightarrow\mathrm{Sym}_{A}(\mathcal{Q},A)$ is given by the
\emph{Chevalley-Eilenberg formula}:
\begin{align*}
&  (d_{k}\omega)(q_{1},\ldots,q_{r+k-1})\\
&  :=\sum_{\sigma\in S_{k,r}}(-)^{\omega(q_{\sigma(1)}+\cdots+q_{\sigma(k)}%
)}\alpha(\sigma,\boldsymbol{q})\{q_{\sigma(1)},\ldots,q_{\sigma(k)}%
\,|\,\omega(q_{\sigma(k)},\ldots,q_{\sigma(k+r)})\}\\
&  \quad\ -\sum_{\tau\in S_{k+1,r-1}}(-)^{\omega}\alpha(\tau,\boldsymbol{q}%
)\omega(\{q_{\tau(1)},\ldots,q_{\tau(k+1)}\},q_{\tau(k+1)},\ldots
,q_{\tau(k+r)}),
\end{align*}
$\omega\in\mathrm{Sym}_{A}^{r}(\mathcal{Q},A)$, $q_{1},\ldots,q_{r+k}%
\in\mathcal{Q}$.
\end{lemma}

\begin{proof}
Let $\omega\in\mathrm{Sym}_{A}^{r}(\mathcal{Q},A)$ be of the form
$\omega=\omega_{1}\cdots\omega_{r}$, $\omega_{1}\in\mathcal{Q}^{\ast}$, and
let $q\in\mathcal{Q}$ be even. Then
\begin{align*}
(d_{k}\omega)(q^{r+k})  &  =\sum_{i=1}^{r}(-)^{\chi}(\omega_{1}\cdots
\widehat{\omega_{i}}\cdots\omega_{r}\cdot d_{k}\omega_{i})(q^{r+k})\\
&  =\tfrac{(k+r)!}{(k+1)!}\sum_{i=1}^{r}(-)^{\chi}\omega_{1}(q)\cdots
\widehat{\omega_{i}(q)}\cdots\omega_{r}(q)(d_{k}\omega_{i})(q^{k+1})\\
&  =\tfrac{(k+r)!}{(k+1)!}\sum_{i=1}^{r}(-)^{\chi}\omega_{1}(q)\cdots
\widehat{\omega_{i}(q)}\cdots\omega_{r}(q)((k+1)\{q^{k}|\omega_{i}%
(q)\}+(-)^{\omega_{i}}\omega_{i}(\{q^{k+1}\}))\\
&  =\tfrac{(k+r)!}{k!}\{q^{k}|\omega_{1}(q)\cdots\omega_{r}(q)\}-(-)^{\omega
}\tbinom{k+r}{k+1}\omega(\{q^{k+1}\},q^{r-1})\\
&  =\tbinom{k+r}{k}\{q^{k}|\omega(q^{r})\}-(-)^{\omega}\tbinom{k+r}{k+1}%
\omega(\{q^{k+1}\},q^{r-1}),
\end{align*}
where $\chi=\bar{\omega}_{1}+\cdots+\bar{\omega}_{i-1}+(\bar{\omega}%
_{i}+1)(\bar{\omega}_{i+1}+\cdots+\bar{\omega}_{r})$.
\end{proof}

\begin{lemma}
\label{Alemma4}Let $a\in A$, $\omega\in\mathcal{Q}^{\ast}$, and $q_{1}%
,\ldots,q_{k+1}\in\mathcal{Q}$. Then%
\[
(E_{k}a)(q_{1},\ldots,q_{k})=(-)^{a(q_{1}+\cdots+q_{k-1})}J^{k}(q_{1}%
,\ldots,q_{k-1}\,|\,a)=0,
\]
and
\[
(E_{k}\omega)(q_{1},\ldots,q_{k+1})=\sum_{i=1}^{k+1}(-)^{\chi}J^{k+1}%
(q_{1},\ldots,\widehat{q_{i}},\ldots,q_{k+1}\,|\,\omega(q_{i}))-\omega
(J^{k+1}(q_{1},\ldots,q_{k+1})),
\]
where
\[
\chi:=\bar{\omega}\sum_{j\neq i}\bar{q}_{j}+\bar{q}_{i}\sum_{j>i}\bar{q}_{j}.
\]

\end{lemma}

\begin{proof}
Let $q\in\mathcal{Q}$ be even. Then
\begin{align*}
(E_{k}a)(q^{k})  &  =\sum_{\ell+m=k}(d_{\ell}d_{m}a)(q^{\ell+m})\\
&  =\sum_{\ell+m=k}\tbinom{\ell+m}{\ell}\{q^{\ell}|(d_{m}a)(q^{m}%
)\}+\sum_{\ell+m=k}\tbinom{\ell+m}{\ell+1}(-)^{a}(d_{m}a)(\{q^{\ell
+1}\},q^{m-1})\\
&  =\sum_{\ell+m=k}\tbinom{\ell+m}{\ell}\{q^{\ell}|\{q^{m}|a\}\}+\sum
_{\ell+m=k}\tbinom{\ell+m}{\ell+1}(-)^{a}\{\{q^{\ell+1}\},q^{m-1}|a\}\\
&  =J^{k+1}(q^{k}|a).
\end{align*}
Similarly,
\begin{align*}
(E_{k}\omega)(q^{k+1})  &  =\sum_{\ell+m=k}(d_{\ell}d_{m}\omega)(q^{\ell
+m+1})\\
&  =\sum_{\ell+m=k}\tbinom{\ell+m+1}{\ell}\{q^{\ell}|d_{m}\omega
(q^{m+1})\}+\sum_{\ell+m=k}\tbinom{\ell+m}{\ell+1}(-)^{\omega}(d_{m}%
\omega)(\{q^{\ell+1}\},q^{m})\\
&  =\sum_{\ell+m=k}\left(  \tbinom{\ell+m+1}{\ell}(m+1)\{q^{\ell}%
|\{q^{m}|\omega(q)\}\}-(-)^{\omega}\tbinom{\ell+m+1}{\ell}\{q^{\ell}%
|\omega(\{q^{m+1}\})\}\right. \\
&  \quad\ +(-)^{\omega}\tbinom{\ell+m}{\ell+1}\{q^{m}|\omega(\{q^{\ell
+1}\})\}+\tbinom{\ell+m}{\ell+1}m\{\{q^{\ell+1}\},q^{m-1}|\omega(q)\}\\
&  \quad\ \left.  -\tbinom{\ell+m}{\ell+1}\omega(\{\{q^{\ell+1}\},q^{m}%
\})\right) \\
&  =-\omega(J^{k+1}(q^{k+1}))+(k+1)J^{k+1}(q^{k}|\omega(q)).
\end{align*}

\end{proof}

\begin{lemma}
\label{Alemma5}Let $\omega\in\mathcal{Q}^{\ast}$ with $p_{1},\ldots,p_{\ell
}\in\mathcal{P}$. Then the expression $R_{\ell}(\omega)$, inductively defined
by
\begin{align*}
R_{\ell}(\omega)(p_{1},\ldots,p_{\ell})  &  :=\Psi_{\ell}(\omega)(p_{1}%
,\ldots,p_{\ell})\\
&  \quad\ -\sum_{\substack{i+j=\ell\\j>0}}\sum_{\sigma\in S_{i,j}}%
\alpha(\sigma,\boldsymbol{p})\psi_{j}(R_{i}(\omega)(p_{\sigma(1)}%
,\ldots,p_{\sigma(i)})))(p_{\sigma(i+1)},\ldots,p_{\sigma(i+j)}),
\end{align*}
is $A$-linear in its argument $\omega$.
\end{lemma}

\begin{proof}
Let $p_{1}=\cdots=p_{\ell}=p$ with $p$ even. I use induction on $\ell$.
$\Psi_{1}$ is $A$-linear, indeed,
\[
\Psi_{1}(a\omega)=\psi_{0}(a)\Psi_{1}(\omega)+\psi_{1}(a)\Psi_{0}%
(\omega)=a\Psi_{1}(\omega).
\]
This provides the base of induction. Now compute%
\begin{align*}
R_{\ell+1}(a\omega)  &  =\Psi_{\ell+1}(a\omega)(p^{\ell+1})-\sum
_{\substack{i+j=\ell+1\\j>0}}\tbinom{i+j}{j}\psi_{j}(R_{\ell}(a\omega
))(p^{\ell+1})\\
&  =a\Psi_{\ell+1}(\omega)(p^{\ell+1})+\sum_{\substack{i+j=\ell+1\\j>0}%
}\tbinom{i+j}{j}\psi_{j}(a)(p^{j})\Psi_{i}(\omega)(p^{i})-\sum
_{\substack{i+j=\ell+1\\j>0}}\tbinom{i+j}{j}\psi_{j}(aR_{i}(\omega))(p^{j})\\
&  =aR_{\ell+1}(\omega)+\sum_{\substack{i+j=\ell+1\\j>0}}\tbinom{i+j}{j}%
\psi_{j}(a)(p^{j})\Psi_{i}(\omega)(p^{i})\\
&  \quad\ -\sum_{\substack{i+j+k=\ell+1\\k>0}}\tbinom{i+j+k}{j+k}\tbinom
{j+k}{j}\psi_{k}(a)(p^{k})\psi_{j}(R_{i}(\omega))(p^{j}).
\end{align*}
The last two summands cancel. Indeed, they are
\begin{align*}
&  \sum_{\substack{i+j=\ell+1\\j>0}}\tbinom{i+j}{j}\psi_{j}(a)(p^{j})\Psi
_{i}(\omega)(p^{i})-\sum_{\substack{i+j+k=\ell+1\\k>0}}\tbinom{i+j+k}%
{j+k}\tbinom{j+k}{j}\psi_{k}(a)(p^{k})\psi_{j}(R_{i}(\omega))(p^{j})\\
&  =\sum_{\substack{i+k=\ell+1\\k>0}}\tbinom{i+k}{k}\psi_{k}(a)(p^{k})\Psi
_{i}(\omega)(p^{i})-\sum_{\substack{i+j+k=\ell+1\\k>0}}\tbinom{i+k}{k}\psi
_{k}(a)(p^{k})R_{i}(\omega)\\
&  \quad\ -\sum_{\substack{i+j+k=\ell+1\\k>0}}\tbinom{i+j+k}{k}\tbinom{i+j}%
{j}\psi_{k}(a)(p^{k})\psi_{j}(R_{i}(\omega))(p^{j})\\
&  =\sum_{\substack{i+k=\ell+1\\k>0}}\tbinom{i+k}{k}\psi_{k}(a)(p^{k}%
)\cdot\lbrack\Psi_{i}(\omega)(p^{i})-R_{i}(\omega)-\sum_{\substack{s+j=i\\j>0}%
}\tbinom{s+j}{j}\psi_{j}(R_{s}(\omega))(p^{j})]\\
&  =0.
\end{align*}

\end{proof}

Now, I present the proof of Lemma \ref{Lemma} (see Section \ref{MSHLR} for the statement).

\begin{proof}
[Proof of Lemma \ref{Lemma}]Let $p_{1}=\cdots=p_{k}=p$ be even. For $r=0$ the
result follows immediately from the definition of $\phi_{k}$. Thus, let $r>0$,
and $\omega=\sigma_{1}\cdots\sigma_{r}$, with $\sigma_{i}\in\mathcal{Q}^{\ast
}$. Then
\begin{align*}
\psi(\omega)_{k}(p^{k})  &  =\psi(\sigma_{1}\cdots\sigma_{r})_{k}(p^{k})\\
&  =\sum_{m_{1}+\cdots+m_{r}=k}(\psi(\sigma_{1})_{m_{1}}\cdots\psi(\sigma
_{r})_{m_{r}})(p^{k})\\
&  =\sum_{m_{1}+\cdots+m_{r}=k}\tbinom{m_{1}+\cdots+m_{r}}{m_{1},\cdots,m_{r}%
}\psi(\sigma_{1})_{m_{1}}(p^{m_{1}})\cdots\psi(\sigma_{r})_{m_{r}}(p^{m_{r}}).
\end{align*}
Now, it follows from (\ref{9}) that, if $\sigma\in\mathcal{Q}^{\ast}$, then
\[
\psi(\sigma)_{m}(p^{m})=\Psi_{m}(\sigma)(p^{m})=\sum_{s+t=m}\tbinom{s+t}%
{s}\psi_{s}(\sigma(\Phi_{t}(p^{t})))(p^{s})=\sum_{s+t=m}\tbinom{s+t}{s}%
\psi(\sigma(\Phi_{t}(p^{t})))_{s}(p^{s}),
\]
so that
\begin{align*}
&  \psi(\omega)_{k}(p^{k})\\
&  =\sum_{s_{1}+t_{1}+\cdots+s_{r}+t_{r}=k}\tbinom{s_{1}+t_{1}+\cdots
+s_{r}+t_{r}}{s_{1},\cdots,s_{r},t_{1},\cdots,t_{r}}\psi(\sigma_{1}%
(\Phi_{t_{1}}(p^{t_{1}})))_{s_{1}}(p^{s_{1}})\cdots\psi(\sigma_{r}(\Phi
_{t_{r}}(p^{t_{r}})))_{s_{r}}(p^{s_{r}})\\
&  =\sum_{s+t_{1}+\cdots+t_{r}=k}\tbinom{s+t_{1}+\cdots+t_{r}}{s,t_{1}%
,\cdots,t_{r}}\psi(\sigma_{1}(\Phi_{t_{1}}(p^{t_{1}}))\cdots\sigma_{r}%
(\Phi_{t_{r}}(p^{t_{r}})))_{s}(p^{s})\\
&  =\sum_{\substack{t_{0}+\cdots+t_{r}=k\\t_{1}\leq\cdots\leq t_{r}}%
}C(t_{0}|t_{1},\ldots,t_{r})\phi_{t_{0}}(p^{t_{0}}|\omega(\Phi_{t_{1}%
}(p^{t_{1}}),\cdots,\Phi_{t_{r}}(p^{t_{r}})),
\end{align*}
where $C(t_{0}|t_{1},\ldots,t_{r})$ is the cardinality of $T_{t_{0}%
|t_{1},\ldots,t_{r}}$.
\end{proof}

\begin{lemma}
\label{Alemma6}Let $\omega\in\mathcal{Q}^{\ast}$, and $p\in\mathcal{P}$ even.
Then $\psi(D_{\mathcal{Q}}\omega)_{k}=(D_{\mathcal{P}}\psi(\omega))_{k}$ iff
\begin{equation}
\sum_{m+\ell=k}\tbinom{m+\ell}{m}\phi_{\ell}(p^{\ell}|\omega(K_{\Phi}%
^{m}(p^{m})))=0, \label{29}%
\end{equation}

\end{lemma}

\begin{proof}
Compute
\begin{align}
&  \psi(D_{\mathcal{Q}}\omega)_{k}(p^{k})\nonumber\\
&  =\sum_{m}\psi(d_{m}\omega)_{k}(p^{k})\nonumber\\
&  =\sum_{m}\sum_{\substack{\ell_{0}+\cdots+\ell_{m+1}=k\\\ell_{1}\leq
\cdots\leq\ell_{m+1}}}C(\ell_{0}|\ell_{1},\ldots,\ell_{m+1})\phi_{\ell_{0}%
}(p^{\ell_{0}}|(d_{m}\omega)(\Phi_{\ell_{1}}(p^{\ell_{1}}),\ldots,\Phi
_{\ell_{m+1}}(p^{\ell_{m+1}})))\nonumber\\
&  =\sum_{m}\sum_{\substack{\ell_{0}+\cdots+\ell_{m+1}=k\\\ell_{1}\leq
\cdots\leq\ell_{m+1}}}C(\ell_{0}|\ell_{1},\ldots,\ell_{m+1})\phi_{\ell_{0}%
}(p^{\ell_{0}}|\{\Phi_{\ell_{1}}(p^{\ell_{1}}),\ldots,\widehat{\Phi_{\ell_{i}%
}(p^{\ell_{i}})},\ldots,\Phi_{\ell_{m+1}}(p^{\ell_{m+1}})|\omega(\Phi
_{\ell_{i}}(p^{\ell_{i}}))\})\nonumber\\
&  \quad\ -\sum_{m}\sum_{\substack{\ell_{0}+\cdots+\ell_{m+1}=k\\\ell_{1}%
\leq\cdots\leq\ell_{m+1}}}(-)^{\omega}C(\ell_{0}|\ell_{1},\ldots,\ell
_{m+1})\phi_{\ell_{0}}(p^{\ell_{0}}|\omega(\{\Phi_{\ell_{1}}(p^{\ell_{1}%
}),\ldots,\Phi_{\ell_{m+1}}(p^{\ell_{m+1}})\})). \label{12}%
\end{align}
Moreover,
\begin{align}
&  (D_{\mathcal{P}}\psi(\omega))_{k}(p^{k})\nonumber\\
&  =\sum_{m+n=k}d_{m}\Psi_{n}(\omega)(p^{k})\nonumber\\
&  =\sum_{m+n=k}\,\tbinom{k}{m}\{p^{m}\,|\Psi_{n}(\omega)(p^{n})\}-\sum
_{m+n=k}(-)^{\omega}\tbinom{k}{m+1}\Psi_{n}(\omega)(\{p^{m+1}\},p^{n-1}%
)\nonumber\\
&  =\sum_{m+\ell+r=k}\,\tbinom{m+\ell+r}{m,\ell,r}[\{p^{m}\,|\phi_{\ell
}(p^{\ell}|\omega(\Phi_{r}(p^{r}))\}\ -\phi_{\ell}(\{p^{m}\},p^{\ell}%
|\omega(\Phi_{r}(p^{r})))\nonumber\\
&  \quad\ -(-)^{\omega}\phi_{\ell}(p^{\ell}|\omega(\Phi_{r}(\{p^{m}%
\},p^{r})))]. \label{13}%
\end{align}
Now, (\ref{11}) implies that the first summand in (\ref{12}) equals the first
two summands in (\ref{13}). It follows that
\begin{align}
&  \sum_{m}\sum_{\substack{\ell_{0}+\cdots+\ell_{m+1}=k\\\ell_{1}\leq
\cdots\leq\ell_{m+1}}}C(\ell_{0}|\ell_{1},\ldots,\ell_{m+1})\phi_{\ell_{0}%
}(p^{\ell_{0}}|\omega(\{\Phi_{\ell_{1}}(p^{\ell_{1}}),\ldots,\Phi_{\ell_{m+1}%
}(p^{\ell_{m+1}})\}))\nonumber\\
&  =\sum_{m+\ell+r=k}\,\tbinom{m+\ell+r}{m,\ell,r}\phi_{\ell}(p^{\ell}%
|\omega(\Phi_{r}(\{p^{m}\},p^{r}))), \label{24}%
\end{align}
that can be compactly rewritten in the form (\ref{29}).
\end{proof}

\section{Higher Lie-Rinehart Brackets as Derived Brackets}

The Lie brackets in a Lie algebroid can be presented as derived brackets
\cite{k04}. Similarly, the higher brackets in an $L_{\infty}$-algebra can be
presented as higher derived brackets following T.{} Voronov \cite{v05} (see
\cite{fz12}). Marco Zambon suggested to me that a similar statement could hold
for the higher brackets in a SH Lie-Rinehart algebra. This is indeed the case
as I briefly show in this appendix. First, I recall the formalism of
\cite{v05}.

\begin{definition}
Let $%
\mathscr{V}%
$ be a DG Lie algebra (over a field $K$ of zero characteristic). Consider the
following data:

\begin{enumerate}
\item an Abelian sub-algebra $%
\mathscr{A}%
\subset%
\mathscr{V}%
$,

\item a $K$-linear projector $\mathbbm{P}:%
\mathscr{V}%
\longrightarrow%
\mathscr{A}%
$,

\item a degree $1$ element $D\in%
\mathscr{V}%
$.
\end{enumerate}

The data $(%
\mathscr{A}%
,\mathbbm{P},D)$ are called $V$\emph{-data in }$%
\mathscr{V}%
$ if

\begin{enumerate}
\item $\ker\mathbbm{P}\subset%
\mathscr{V}%
$ is a Lie subalgebra,

\item $\mathbbm{P}D=0$,

\item $D^{2}=0$.
\end{enumerate}
\end{definition}

A triple $(%
\mathscr{A}%
,\mathbbm{P},D)$ of $V$-data determines, in particular, $K$-multilinear,
graded-symmetric maps
\[
\{{}\cdot{},\cdots,{}\cdot{}\}_{k}^{D}:{}%
\mathscr{A}%
^{\times k}\longrightarrow%
\mathscr{A}%
,\quad k\in\mathbb{N}%
\]
via
\[
\{a_{1},a_{2},\cdots,a_{k}\}^{D}:=\mathbbm{P}[[\cdots\lbrack\lbrack
D,a_{1}],a_{2}]\cdots],a_{k}],\quad a_{1},a_{2},\ldots,a_{k}\in%
\mathscr{A}%
.
\]

\begin{theorem}
[Voronov \cite{v05}]\label{TheorV}The brackets $\{{}\cdot{},\cdots,{}\cdot
{}\}_{k}^{D}$, $k\in\mathbb{N}$, give $%
\mathscr{V}%
$ the structure of an $L_{\infty}[1]$-algebra.
\end{theorem}

Now, let $A$ be a graded algebra and $\mathcal{Q}$ a projective and finitely
generated $A$-module. As in Section \ref{SHDGLR}, consider the graded algebra
$\mathcal{A}:=\mathrm{Sym}_{A}(\mathcal{Q},A)$. Let $\mathrm{Der}\mathcal{A}$
be the graded Lie algebra of derivations of $\mathcal{A}$, and let $%
\mathscr{V}%
:=\mathrm{Der}\mathcal{A}\oplus\mathcal{A}$ be the graded Lie algebra of first
order differential operators in $\mathcal{A}$. The brackets in $%
\mathscr{V}%
$ are
\[
\lbrack(\Delta,\omega),(\nabla,\rho)]=([\Delta,\nabla],\Delta\rho
-(-)^{\omega\nabla}\nabla\omega),\quad\Delta,\nabla\in\mathrm{Der}%
\mathcal{A},\quad\omega,\rho\in\mathcal{A}.
\]
Notice that $\mathcal{Q}\oplus A$ embeds into $%
\mathscr{V}%
$. Namely, consider the linear map
\[
\mathbbm{i}:\mathcal{Q}\oplus A\ni(q,a)\longmapsto(\mathbbm{i}_{q}%
,\mathbbm{i}_{a})\in%
\mathscr{V}%
,
\]
where $\mathbbm{i}_{a}:=-a$,
and $\mathbbm{i}_{q}$ is defined by%
\[
\mathbbm{i}_{q}\omega:=-(-)^{q\omega}\omega(q,\cdot,\cdots,\cdot),\quad
\omega\in\mathcal{A}.
\]
The signs are chosen to simplify formulas below. The image $%
\mathscr{A}%
$ of $\mathbbm{i}$ is clearly an Abelian subalgebra in $%
\mathscr{V}%
$. There is a canonical projection
\[
\mathbbm{P}:%
\mathscr{V}%
\ni(\Delta,\omega)\longmapsto(\mathbbm{P}_{\mathcal{Q}}\Delta,\mathbbm{P}_{A}%
\omega)\in\mathcal{Q}\oplus A,
\]
where $\mathbbm{P}_{A}\omega:=-p\omega$, $p:\mathcal{A}\longrightarrow A$
being the natural projection, and $\mathbbm{P}_{\mathcal{Q}}$ is implicitly
defined by the formula
\[
\omega(\mathbbm{P}_{\mathcal{Q}}\Delta):=-(-)^{\Delta\omega}p\Delta\omega\in
A,\quad\omega\in\mathcal{Q}^{\ast}.
\]
Notice that the expression $(-)^{\Delta\omega}p\Delta\omega$ is graded,
$A$-linear in $\omega$ so that $\mathbbm{P}_{\mathcal{Q}}\Delta$ is a well
defined element in $\mathcal{Q}$. Obviously, $\mathbbm{P}$ is a left inverse
of $\mathbbm{i}$. Finally, it is easy to see that $\ker\mathbbm{P}\subset$ $%
\mathscr{V}%
$ is a subalgebra as well, so that, $(%
\mathscr{A}%
,\mathbbm{P},D)$ \emph{are }$V$\emph{-data in }$%
\mathscr{V}%
$\emph{ for all homological derivations }$D\in\mathrm{Der}\mathcal{A}$. In
particular, a homological derivation $D$ in $\mathcal{A}$ determines an
$L_{\infty}[1]$-algebra structure $\mathscr{L}^{D}:=\{\{\cdot,\cdots
,\cdot\}_{k}^{D},\ k\in\mathbb{N}\}$ in $A\oplus\mathcal{Q}$ via Theorem
\ref{TheorV}. On the other hand, a homological derivation in $\mathcal{A}$
determines an $L_{\infty}[1]$-algebra structure $\mathscr{L}^{\oplus
}:=\{\{\cdot,\cdots,\cdot\}_{k}^{\oplus},\ k\in\mathbb{N}\}$ in $A\oplus
\mathcal{Q}$ also via Corollary \ref{Cor} (see Definition \ref{Def2} for
details about the brackets $\{\cdot,\cdots,\cdot\}^{\oplus}$ determined by the
brackets in $\mathcal{Q}$ and the anchors).

\begin{proposition}
Let $D\in\mathrm{Der}\mathcal{A}$ be a homological derivation. The $L_{\infty
}[1]$-algebra structures $\mathscr{L}^{D}$ and $\mathscr{L}^{\oplus}$ coincide.
\end{proposition}

\begin{proof}
First of all, since $[D,a]=Da$ for all $a\in A$, it is obvious that
$\{v_{1},\ldots,v_{k}\}_{k}^{D}$ vanishes whenever two entries are from $A$.
This means that $\mathscr{L}^{D}$ is the same as an $L_{\infty}[1]$-algebra
structure on $\mathcal{Q}$, and an $L_{\infty}[1]$-module structure on $A$.
Now, define derivations $d_{k}:\mathcal{A}\longrightarrow\mathcal{A}$ as in
the proof of Corollary \ref{Cor}, and, for $a\in A$ and $q\in\mathcal{Q}$
even, compute,
\begin{align*}
\{q^{k-1},a\}^{D}  &
=\mathbbm{P}[[[\cdots\lbrack D,\mathbbm{i}_{q}]\cdots],\mathbbm{i}_{q}],\mathbbm{i}_{a}%
]\\
&  =-(-)^{k}p\mathbbm{i}_{q}^{k-1}Da\\
&  =(d_{k-1}a)(q^{k-1})\\
&  =\{q^{k-1}|a\}\\
&  =\{q^{k-1},a\}^{\oplus}.
\end{align*}
Similarly, let $\omega\in\mathcal{Q}^{\ast}$ and compute
\begin{align*}
\omega(\{q^{k}\}^{D})  &  =\omega
(\mathbbm{P}[[\cdots\lbrack D,\mathbbm{i}_{q}]\cdots],\mathbbm{i}_{q}])\\
&  =-(-)^{\omega}p[[\cdots\lbrack D,\mathbbm{i}_{q}]\cdots],\mathbbm{i}_{q}%
](\omega)\\
&  =-(-)^{\omega+k}p(\mathbbm{i}_{q}^{k}D\omega-k\mathbbm{i}_{q}%
^{k-1}D\mathbbm{i}_{q}\omega)\\
&  =-(-)^{\omega}((d_{k-1}\omega)(q^{k})-kd_{k-1}(\omega(q))(q^{k-1}))\\
&  =\omega(\{q^{k}\})\\
&  =\omega(\{q^{k}\}^{\oplus}).
\end{align*}

\end{proof}

Notice that when $\mathcal{Q}$ is as in Section \ref{SHLRF} and $D$ is the de
Rham differential in $\mathcal{A}=\Lambda(M)$, then the $V$-data
$(\mathscr{A},\mathbbm{P},D)$ in $\mathscr{V}$ are precisely those constructed
by Ji \cite{j12} in the case of a foliation.

\section{The Homotopy Lie-Rinehart Algebra of a Foliation via Homotopy
Transfer}

\label{Ap3}

After the publication on arXiv of a preliminary version of this paper, Florian
Sch\"{a}tz suggested to me that the $L_{\infty}[1]$-algebra of a foliation
could be derived from the DG Lie algebra $\mathrm{Der}\overline{\Lambda}$ of
derivations of $\overline{\Lambda}$ via homotopy transfer (see \cite{h10}
about the homotopy trasfer of Lie algebra structures). This is indeed the case
as I briefly discuss in this appendix. I first recall the version of the
homotopy transfer theorem I will refer to.

\begin{theorem}
[Homotopy Transfer Theorem]Let $(L,\Delta,[{}\cdot{},{}\cdot{}])$ be a DG Lie
algebra over a field of $0$ characteristic, and
\[%
\xymatrix{     *{ \quad\ \  \quad(L, \Delta)\ }
\ar@(dl,ul)[]^{h}\
\ar@<0.5ex>[r]^{p} & *{\
(H,\delta)\quad\ \  \ \quad}  \ar@<0.5ex>[l]^{j}}%
\]
contraction data, i.e., i) $p$ and $j$ are cochain maps, ii) $p\circ
j=\mathrm{id}_{H}$, and iii) $\mathrm{id}_{L}-j\circ p=\Delta\circ
h+h\circ\delta$. Then there is a natural $L_{\infty}$-algebra structure
$\mathscr{L}$ on $(H,\delta)$.

\end{theorem}

There exists an explicit description of brackets in $(H,\mathscr{L})$ in terms
of the contraction data by means of trees \cite{ks01}, or inductive formulas
(see for instance \cite{m99}, where {inductive formulas for the transfer of an
associative algebra structure are provided}).

I'm not presenting here this description (the interested reader may see
\cite{v12}, where I recall the necessary formulas from \cite{m99} and apply
them to prove the existence of more SH structures associated to a foliation).
Notice, however, that, in a similar way, one can transfer the structure of a
DG Lie module along contraction data, and get an $L_{\infty}$-module.

Now, $\mathrm{Der}\overline{\Lambda}$ possesses the canonical differential
$\Delta:=[\overline{d},{}\cdot{}]$ and $(\mathrm{Der}\overline{\Lambda}%
,\Delta,[{}\cdot{},{}\cdot{}])$ is a DG Lie algebra (sometimes referred to as
the \emph{deformation complex of the Lie algebroid }$\overline{\mathfrak{X}}$
\cite{cm08}). It is known that $(\mathrm{Der}\overline{\Lambda},\Delta)$ is
homotopy equivalent to $(\overline{\Lambda}\otimes\overline{\mathfrak{X}%
},\overline{d})$ \cite{cm08}. However, to my knowledge, there was as yet no
explicit description of contraction data. I provide it in the proof of the next

\begin{proposition}
A distribution $V$ complementary to $C$ determines contraction data
\[%
\xymatrix{     *{ \quad\ \  \quad(\mathrm{Der}\overline{\Lambda}, \Delta)\ }
\ar@(dl,ul)[]+<-3ex,-1ex>;[]+<-3ex,1ex>^-h \ar@<0.5ex>[r]^-{p} & *{\
(\overline{\Lambda}\otimes\mathfrak{X},\overline{d})\quad\ \  \ \quad}
\ar@<0.5ex>[l]^-{j}}%
.
\]

\end{proposition}

\begin{proof}
For $Z\in\overline{\Lambda}\otimes\overline{\mathfrak{X}}$, and $\lambda
\in\overline{\Lambda}$ put
\[
j(Z)(\lambda):=\overline{L_{Z}\lambda}.
\]
Then $j(Z)\in\mathrm{Der}\overline{\Lambda}$. For $\mathcal{D}\in
\mathrm{Der}\overline{\Lambda}$, put $p(\mathcal{D}):=\overline{\mathcal{D}%
|_{C^{\infty}(M)}}\in\overline{\Lambda}\otimes\overline{\mathfrak{X}}$.
Finally, let $h(\mathcal{D})\in\mathrm{Der}\overline{\Lambda}$ be defined on
generators $f$, $\overline{d}f$ by
\begin{align*}
h(\mathcal{D})(f)  &  :=0\\
h(\mathcal{D})(\overline{d}f)  &  :=(-)^{\mathcal{D}}(\mathcal{D}%
-jp\mathcal{D})f,
\end{align*}
It is easy to see, using, for instance, local coordinates, that $h(\mathcal{D}%
)$ is well defined, and $j,p,h$ are actually contraction data.
\end{proof}

As an immediate corollary of the Homotopy Transfer Theorem and the above
proposition, there is an $L_{\infty}[1]$-algebra structure on $\overline
{\Lambda}\otimes\overline{\mathfrak{X}}[1]$. It is easy to see
that such $L_{\infty}[1]$-algebra actually coincides with the one described in
Section \ref{SHLRF}. Notice that the $L_{\infty}[1]$-module structure on
$\overline{\Lambda}$ can be obtained from homotopy transfer as well.
%

\section{Alternative Formulas for Binary Operations}

\label{Ap4}

Let $(A,\mathcal{Q})=(\overline{\Lambda},\overline{\Lambda}\otimes\overline{\mathfrak{X}}[1])$ denote
the $LR_{\infty}[1]$-algebra of a foliation. In this appendix I present
alternative formulas for the binary operations in $\mathcal{Q}$. This is
useful for some purposes, e.g., proving the homotopy transfer and the derived
bracket \cite{j12} origins of $\mathcal{Q}$.

\begin{proposition}
Let $Z_{1},Z_{2}\in\mathcal{Q}$. Then
\[
\{Z_{1},Z_{2}\}=-(-)^{Z_{1}}\overline{[\![Z_{1},Z_{2}]\!]}.
\]

\end{proposition}

\begin{proof}
Let $X\in\overline{\mathfrak{X}}$ and $\lambda\in\overline{\Lambda}$. Then
$i_{X}\lambda=0$ and $L_{X}\lambda=\lambda^{\prime}+\lambda^{\prime\prime}$,
with $\lambda^{\prime}\in\overline{\Lambda}$ and $\lambda^{\prime\prime}%
\in\overline{\Lambda}\otimes C\Lambda^{1}$. Therefore, in view of Formula
(\ref{f2}),
\[
\lbrack\![Z{}_{1},Z{}_{2}]\!]=\overline{[\![Z{}_{1},Z{}_{2}]\!]}+Z^{\prime
}+Z^{\prime\prime}%
\]
with $Z^{\prime}\in\overline{\Lambda}\otimes C\mathfrak{X}$, $Z^{\prime\prime
}\in C\Lambda^{1}\otimes\overline{\Lambda}\otimes\overline{\mathfrak{X}}$. It
follows that
\[
Z^{\prime}=i_{[\![Z{}_{1},Z{}_{2}]\!]}P^{C}\text{\quad and\quad}%
Z^{\prime\prime}=i_{P^{V}}[\![Z{}_{1},Z{}_{2}]\!],
\]
so that
\begin{align*}
\overline{\lbrack\![Z{}_{1},Z{}_{2}]\!]}  &  =[\![Z{}_{1},Z{}_{2}%
]\!]-Z^{\prime}-Z^{\prime\prime}\\
&  =i_{[\![Z{}_{1},Z{}_{2}]\!]}\mathbb{I-}i_{[\![Z{}_{1},Z{}_{2}]\!]}%
P^{C}-i_{P^{V}}[\![Z{}_{1},Z{}_{2}]\!]\\
&  =i_{[\![Z{}_{1},Z{}_{2}]\!]}P^{V}-i_{P^{V}}[\![Z{}_{1},Z{}_{2}]\!]\\
&  =[[\![Z{}_{1},Z{}_{2}]\!],P^{V}]_{\mathrm{nr}}.
\end{align*}
Now, it follows from Formula (\ref{f1}) that
\begin{align*}
\lbrack\lbrack\![Z{}_{1},Z{}_{2}]\!],P^{V}]_{\mathrm{nr}}  &  =[\![Z{}%
_{1},[Z{}_{2},P^{V}]_{\mathrm{nr}}]\!]-(-)^{Z_{2}(Z_{1}+1)}[Z{}_{2}%
,[\![Z{}_{1},P^{V}]\!]]_{\mathrm{nr}}\\
&  =[\![Z{}_{1},Z{}_{2}]\!]-(-)^{(Z_{1}+1)(Z_{2}+1)}[Z{}_{2},[\![P^{C},Z{}%
_{1}]\!]]_{\mathrm{nr}}\\
&  =[\![Z{}_{1},Z{}_{2}]\!]-(-)^{(Z_{1}+1)(Z_{2}+1)}[Z{}_{2},\overline{d}%
Z{}_{1}+[R,Z{}_{1}]_{\mathrm{nr}}]_{\mathrm{nr}}\\
&  =[\![Z{}_{1},Z{}_{2}]\!]-(-)^{(Z_{1}+1)(Z_{2}+1)}[Z{}_{2},[R,Z{}%
_{1}]_{\mathrm{nr}}]_{\mathrm{nr}}\\
&  =[\![Z{}_{1},Z{}_{2}]\!]-(-)^{Z_{1}}[[R,Z{}_{1}]_{\mathrm{nr}},Z{}%
_{2}]_{\mathrm{nr}},
\end{align*}
where I also used that $[Z{},Z{}_{1}]_{\mathrm{nr}}=0$ for all $Z{},Z{}_{1}%
\in\mathcal{Q}$.
\end{proof}

\begin{proposition}
Let $Z\in\mathcal{Q}$ and $\lambda\in\overline{\Lambda}$. Then
\[
\{Z|\lambda\}=-(-)^{Z_{1}}\overline{L_{Z}\lambda}.
\]

\end{proposition}

\begin{proof}
In view of Formula (\ref{f4})
\[
L_{Z}\lambda=\overline{L_{Z}\lambda}+\omega^{\prime}%
\]
with $\omega^{\prime}\in\overline{\Lambda}\otimes C\Lambda^{1}$. It follows
that
\begin{align*}
\omega^{\prime}  &  =i_{P^{V}}L_{Z}\lambda\\
&  =[i_{P^{V}},L_{Z}]\lambda\\
&  =-i_{[\![Z,P^{V}]\!]}\lambda\\
&  =(-)^{Z}i_{[\![P^{C},Z]\!]}\lambda\\
&  =(-)^{Z}i_{\overline{d}Z+[R,Z]_{\mathrm{nr}}}\lambda\\
&  =(-)^{Z}i_{[R,Z]_{\mathrm{nr}}}\lambda.
\end{align*}
where I used Formula (\ref{f5}).
\end{proof}

\subsection*{Aknowledgements}

I thank Jim Stasheff for carefully reading earlier versions of this paper. The
present version has been strongly influenced by his numerous comments and
suggestions. I also thank Florian Sch\"{a}tz for suggesting me that the
$LR_{\infty}[1]$-algebra of a foliation could arise from homotopy transfer.
Finally, I thank Marco Zambon for having drawn my attention to the role that
the derived bracket construction could play in the description of homotopy
Lie-Rinehart algebras in terms of multi-differential algebras.

\end{document}